\documentclass[10pt]{amsart}
\usepackage{graphicx}
\usepackage{amsmath,amssymb,xy,array}
\usepackage[T1]{fontenc}
\usepackage{amsfonts}
\usepackage{mathabx}
\usepackage{amsmath}
\usepackage{amssymb}
\usepackage{amsthm}
\usepackage{graphicx}
\usepackage{enumitem}
\usepackage{xcolor}
\usepackage[utf8]{inputenc}
\usepackage{hyperref}
\usepackage{mathrsfs}
\usepackage{comment}
\usepackage[toc,page]{appendix}
\usepackage{fancyhdr}
\usepackage{tikz}
\usepackage{tikz-cd}
\usepackage{longtable}
\usepackage{geometry}
\usepackage{textcomp}
\usetikzlibrary{trees}
\usetikzlibrary{arrows}
\usepackage{multirow}
\usepackage{youngtab}
\usepackage{mathdots}

\definecolor{yqyqyq}{rgb}{0.5019607843137255,0.5019607843137255,0.5019607843137255}\definecolor{uuuuuu}{rgb}{0.26666666666666666,0.26666666666666666,0.26666666666666666}
\definecolor{uququq}{rgb}{0.25098039215686274,0.25098039215686274,0.25098039215686274}
\definecolor{wwwwww}{rgb}{0.4,0.4,0.4}
\definecolor{uuuuuu}{rgb}{0.26666666666666666,0.26666666666666666,0.26666666666666666}

\setlist[itemize]{leftmargin=6mm}

\renewcommand{\P}{\mathbb P}

\DeclareMathOperator{\codim}{codim}

\newcommand{\Aut}{\operatorname{Aut}}
\newcommand{\PsAut}{\operatorname{PsAut}}

\DeclareMathOperator{\Cl}{Cl}

\DeclareMathOperator{\NE}{NE}

\DeclareMathOperator{\lin}{lin}

\DeclareMathOperator{\mult}{mult}

\DeclareMathOperator{\Exc}{Exc}

\DeclareMathOperator{\Eff}{Eff}
\DeclareMathOperator{\Nef}{Nef}
\DeclareMathOperator{\Mov}{Mov}
\DeclareMathOperator{\Pic}{Pic}
\DeclareMathOperator{\rank}{rank}

\renewcommand{\sec}{\mathbb{S}ec}

\DeclareMathOperator{\Cox}{Cox}

\DeclareMathOperator{\mov}{mov}
\DeclareMathOperator{\pf}{pf}

\renewcommand{\P}{\mathbb{P}}

\newtheorem{thm}{Theorem}[section]

\newtheorem{Lemma}[thm]{Lemma}
\newtheorem{Proposition}[thm]{Proposition}

\newtheorem{Corollary}[thm]{Corollary}
\newtheorem{conj}[thm]{Conjecture}

\theoremstyle{definition}

\newtheorem{Definition}[thm]{Definition}
\newtheorem{Remark}[thm]{Remark}

\newtheorem{Notation}[thm]{Notation}

\newtheorem{Construction}[thm]{Construction}

\hypersetup{pdfpagemode=UseNone}
\hypersetup{pdfstartview=FitH}

\begin{document}

\title{On the birational geometry of spaces of complete forms II: skew-forms}

\author[Alex Massarenti]{Alex Massarenti}
\address{\sc Alex Massarenti\\ Dipartimento di Matematica e Informatica, Universit\`a di Ferrara, Via Machiavelli 30, 44121 Ferrara, Italy\newline
\indent Instituto de Matem\'atica e Estat\'istica, Universidade Federal Fluminense, Campus Gragoat\'a, Rua Alexandre Moura 8 - S\~ao Domingos\\
24210-200 Niter\'oi, Rio de Janeiro\\ Brazil}
\email{alex.massarenti@unife.it, alexmassarenti@id.uff.br}

\date{\today}

\begin{abstract}
Moduli spaces of complete skew-forms are compactifications of spaces of skew-symmetric linear maps of maximal rank on a fixed vector space, where the added boundary divisor is simple normal crossing. In this paper we compute their effective, nef and movable cones, the generators of their Cox rings, and for those spaces having Picard rank two we give an explicit presentation of the Cox ring. Furthermore, we give a complete description of both the Mori chamber and stable base locus decompositions of the effective cone of some spaces of complete skew-forms having Picard rank at most four.
\end{abstract}

\maketitle 

\setcounter{tocdepth}{1}

\tableofcontents

\section{Introduction}
A \textit{complete skew-form} on a $K$-vector space $V$ is a finite sequence of non-zero $2$-forms $\omega_i$, where $\omega_1$ is a $2$-form on $V$, $\omega_{i+1}$ is a $2$-form on the kernel of $\omega_i$, and the last $2$-form in the sequence is either non-degenerate or has kernel of dimension one. In this paper we study the birational geometry of \textit{moduli spaces of complete skew-forms}, extending the study of moduli spaces of complete collineations and quadrics carried out in \cite{Ma18}. This triad of spaces has been much studied from the end of the $19$-th century up to the present day \cite{Ch64}, \cite{Gi03}, \cite{Hi75}, \cite{Hi77}, \cite{Sc86}, \cite{Se84}, \cite{Se48}, \cite{Se51}, \cite{Se52}, \cite{Ty56}, \cite{Va82}, \cite{Va84}, \cite{TK88}, \cite{LLT89}, \cite{Tha99}, \cite{Ce15}, \cite{Ca16}.

In this paper we investigate some Mori theoretical aspects of the geometry of moduli spaces of complete skew-forms. Recall that the cone of effective divisors $\Eff(X)$ of a Mori dream space $X$ admits a well-behaved decomposition into convex sets, called Mori chambers, and these chambers are the nef cones of birational models of $X$. Mori dream spaces were introduced by Y. Hu and S. Keel in \cite{HK00}, and behave very well with respect to the minimal model program.

We will denote by $\mathcal{A}(n)$ the moduli space of complete skew-forms on an $n+1$ dimensional $K$-vector space $V$. Our investigation starts from a construction of $\mathcal{A}(n)$ due to M. Thaddeus \cite{Tha99}, as a sequence of blow-ups of $\mathbb{P}(\bigwedge^2V)$ along the Grassmannian $\mathcal{G}(1,n)$ of lines in $\mathbb{P}^n$, and then along all the secant varieties of $\mathcal{G}(1,n)$ in order of increasing dimension.  

In Section \ref{sec2}, we study the natural action of $SL(n)$ on $\mathcal{A}(n)$, and as a consequence we compute the generators of its effective and nef cone, while in Section \ref{sec3} we give a minimal set of generators for the Cox ring of $\mathcal{A}(n)$. Recall that Cox rings, first introduced by D. A. Cox for toric varieties \cite{Cox95}, are defined as a direct sum of the spaces of sections of all isomorphism classes of line bundles on a given variety. These algebraic objects encode much information on the birational geometry of a variety such as its Mori chamber decomposition. The main results in Theorems \ref{theff} and \ref{gen} can be summarized as follows.
\begin{thm}\label{main}
Let us denote by $D_{2j+2}$ the strict transform in $\mathcal{A}(n)$ of the divisor in $\mathbb{P}(\bigwedge^2V)$, with homogeneous coordinates $[z_{0,1}:\dots:z_{n-1,n}]$, given by 
\begin{equation*}
\pf\left(
\begin{array}{cccc}
0 & z_{n-2j-1,n-2j} & \dots & z_{n-2j-1,n}\\ 
-z_{n-2j-1,n-2j} & 0 & \dots & z_{n-2j,n}\\
\vdots & \ddots & \ddots & \vdots\\ 
-z_{n-2j-1,n} & \dots & -z_{n-1,n} & 0
\end{array}\right)=0
\end{equation*}
where $\pf$ denotes the pfaffian, and by $E_j$ the exceptional divisors of the blow-ups in Thaddeus's construction. 

If $n$ is odd then $\Eff(\mathcal{A}(n)) = \left\langle E_{1}^{-},\dots,E_{\frac{n-3}{2}}^{-}, D_{n+1}^{-}\right\rangle$ and $\Nef(\mathcal{A}(n)) = \left\langle D_{2j+2}^{-},\: j = 0,\dots,\frac{n-3}{2}\right\rangle$, while if $n$ is even we have $\Eff(\mathcal{A}(n)) = \left\langle E_1^{-},\dots,E_{\frac{n-2}{2}}^{-},D_{n}^{-}\right\rangle$ and $\Nef(\mathcal{A}(n)) = \left\langle D_{2j+2}^{-},\: j = 0,\dots,\frac{n-2}{2}\right\rangle$.

Furthermore, the canonical sections associated to the $D_{2j+2}$ and the $E_j$ form a set of minimal generators of $\Cox(\mathcal{A}(n))$. 

Finally, the Cox rings of $\mathcal{A}(4)$ and $\mathcal{A}(5)$ are isomorphic to the homogeneous coordinate rings respectively of the $10$-dimensional Spinor varieties $\mathbb{S}_{5}\subset\mathbb{P}^{15}$ and of the $15$-dimensional Spinor variety $\mathbb{S}_{6}\subset\mathbb{P}^{31}$.
\end{thm} 

The pseudo-effective cone $\overline{\Eff}(X)$ of a projective variety $X$ with $h^1(X,\mathcal{O}_X)=0$, such as a Mori dream space, can be decomposed into chambers depending on the stable base loci of linear series. Such decomposition called \textit{stable base locus decomposition} in general is coarser than the Mori chamber decomposition.   

In Section \ref{sec4} thanks to the computation of the generators of the Cox rings in Section \ref{sec3} we determine the Mori chamber and the stable base locus decomposition of $\mathcal{A}(n)$ for $n\in\{4,5,6,7,8\}$. Indeed, as a consequence of Theorems \ref{MCD_main}, \ref{MCD_main_2} we have the following.

\begin{thm}
The Mori chamber and stable base locus decompositions of $\Eff(\mathcal{A}(4))$ and $\Eff(\mathcal{A}(5))$ coincide and consist respectively of $2$ and $3$ chambers. The Mori chamber decomposition of $\Eff(\mathcal{A}(6))$ consists of $5$ chambers while its stable base locus decomposition consists of $4$ chambers. Furthermore, the Mori chamber decomposition of $\Eff(\mathcal{A}(7))$ consists of $9$ chambers while its stable base locus decomposition consists of $8$ chambers. Finally, the Mori chamber decomposition of $\Eff(\mathcal{A}(8))$ has $15$ chambers while its stable base locus decomposition has $9$ chambers.
\end{thm}

All these decompositions are described in detail in Theorems \ref{MCD_main}, \ref{MCD_main_2}. In Conjecture \ref{conjcham}, based on these results, we present a conjectural description for the stable base locus decomposition of $\Eff(\mathcal{A}(n))\setminus \Mov(\mathcal{A}(n))$, and we extend such description to the spaces of complete collineations and quadrics in \cite[Constructions 2.4, 2.6]{Ma18}. Finally, in Section \ref{psaut} we give an explicit presentation of the group of pseudo-automorphisms of $\mathcal{A}(n)$.

\subsection*{Organization of the paper}
All through the paper we will work over an algebraically closed field of characteristic zero. In Section \ref{sec1} we recall the basics on moduli spaces of complete skew-forms and their construction as blow-ups. In Section \ref{sec2} we compute their cones of divisors and curves, and in Section \ref{sec3} we compute the generators of their Cox rings. Section \ref{sec4} is devoted to the computation of Mori chamber and stable base locus decompositions.

\subsection*{Acknowledgments}
The author is a member of the Gruppo Nazionale per le Strutture Algebriche, Geometriche e le loro Applicazioni of the Istituto Nazionale di Alta Matematica "F. Severi" (GNSAGA-INDAM). I thank the referee for the helpful comments that helped me to improve the paper.

\section{Moduli spaces of complete skew-forms}\label{sec1}
Let $V$ be a $K$-vector space of dimension $n+1$, and let $\mathbb{P}^{N_{-}}$ with $N_{-} = \binom{n+1}{2}-1$ be the projective space parametrizing non-zero skew-symmetric linear maps $V\rightarrow V$ up to a scalar multiple. Here with the notation $N_{-}$ we want to stress that we are working with the skew-symmetric counterpart of the space of complete collineations.

Recall that given an irreducible and reduced non-degenerate variety $X\subset\P^N$, and a positive integer $h\leq N$ we denote by $\sec_h(X)$ 
the \emph{$h$-secant variety} of $X$. This is the subvariety of $\P^N$ obtained as the closure of the union of all $(h-1)$-planes 
$\langle x_1,...,x_{h}\rangle$ spanned by $h$ general points of $X$. 

A point $p\in \mathbb{P}^{N_{-}} = \mathbb{P}(\bigwedge^2V)$ can be represented by an $(n+1)\times (n+1)$ skew-symmetric matrix
\stepcounter{thm}
\begin{equation}\label{matrix}
Z = \left(
\begin{array}{cccc}
0 & z_{0,1} & \dots & z_{0,n}\\ 
-z_{0,1} & 0 & \dots & z_{1,n}\\
\vdots & \ddots & \ddots & \vdots\\ 
-z_{0,n} & \dots & -z_{n-1,n} & 0
\end{array}\right)
\end{equation}
The Grassmannian $\mathcal{G}(1,n)$ is the locus of rank two matrices. More generally, $p\in \sec_h(\mathcal{G}(1,n))$ if and only if $Z$ can be written as a linear combination of $h$ rank two matrices that is if and only if $\rank(Z)\leq 2h$. Note that this last statement holds for all points of $\sec_h(\mathcal{G}(1,n))$ and not just for a general point. Indeed, if $Z\in \sec_h(\mathcal{G}(1,n))$ we can find a sequence of matrices $\{Z_t\}_{t\in K}$ such that $Z_t$ has rank at most $2h$ and $\lim_{t\mapsto 0}Z_t = Z$. Now, since all $(2h+2)\times (2h+2)$ sub-Pfaffians of $Z_t$ are zero by continuity all $(2h+2)\times (2h+2)$ sub-Pfaffians of $Z$ are zero as well. 

The ideal of $\sec_h(\mathcal{G}(1,n))$ is generated by the $(2h+2)\times (2h+2)$ sub-Pfaffians of $Z$ \cite[Section 10]{LO13}. 

The \textit{space of complete skew-forms} is the closure of the graph of the rational map 
\stepcounter{thm}
\begin{equation}\label{closure}
\begin{array}{ccc}
\mathbb{P}(\bigwedge^2V)& \dasharrow & \mathbb{P}(\bigwedge^2\bigwedge^{2}V)\times\dots\times \mathbb{P}(\bigwedge^2\bigwedge^nV)\\
 Z & \longmapsto & (\wedge^2Z,\dots,\wedge^{n}Z)
\end{array}
\end{equation}
By \cite[Theorem 6.3]{Tha99} this spaces can by realized by the following sequence of blow-ups.

\begin{Construction}\label{css}
Let us consider the following sequence of blow-ups:
\begin{itemize}
\item[-] $\mathcal{A}(n)_1$ is the blow-up of $\mathcal{A}(n)_0:=\mathbb{P}^{N_{-}}$ along the Grassmannian $\mathcal{G}(1,n)$;
\item[-] $\mathcal{A}(n)_2$ is the blow-up of $\mathcal{A}(n)_1$ along the strict transform of $\sec_2(\mathcal{G}(1,n))$;\\
$\vdots$
\item[-] $\mathcal{A}(n)_i$ is the blow-up of $\mathcal{A}(n)_{i-1}$ along the strict transform of $\sec_i(\mathcal{G}(1,n))$;\\
$\vdots$
\item[-] $\mathcal{A}(n)_{\lfloor\frac{n-1}{2}\rfloor}$ is the blow-up of $\mathcal{A}(n)_{\lfloor\frac{n-1}{2}\rfloor-1}$ along the strict transform of $\sec_{\lfloor\frac{n-1}{2}\rfloor}(\mathcal{G}(1,n))$.
\end{itemize}
Let $f_i:\mathcal{A}(n)_i\rightarrow \mathcal{A}(n)_{i-1}$ be the blow-up morphism. We will denote by $E_i^{-}$ both the exceptional divisor of $f_i$ and its strict transforms in the subsequent blow-ups, and by $H^{-}$ the pull-back to $\mathcal{A}(n):=\mathcal{A}(n)_{\lfloor\frac{n-1}{2}\rfloor}$ of the hyperplane section of $\mathbb{P}^{N_{-}}$. We will denote by $f^{-}:\mathcal{A}(n)\rightarrow\mathbb{P}^{N_{-}}$ the composition of the $f_i^{-}$'s. 

Then for any $i = 1,\dots,n$ the variety $\mathcal{A}(n)_{i}$ is smooth, the strict transform of $\sec_i(\mathcal{G}(1,n))$ in $\mathcal{A}(n)_{i-1}$ is smooth, and the divisor $E_1^{-}\cup E_2^-\cup\dots \cup E_{i-1}^-$ in $\mathcal{A}(n)_{i-1}$ is simple normal crossing. Furthermore, the variety $\mathcal{A}(n)$ is isomorphic to the space of complete skew-forms.
\end{Construction}

\begin{Remark}\label{dimsec}
By \cite[Theorem 2.1]{CGG05} the dimension of the secant varieties appearing in Construction \ref{css} is given by 
$$\dim(\sec_h(\mathcal{G}(1,n)) = 2(n-1)h+h-1-2h(h-1)$$
for $h < \lfloor\frac{n+1}{2}\rfloor$.
\end{Remark}

\section{Curves and divisors on spaces of complete skew-forms}\label{sec2}
Let $X$ be a normal projective $\mathbb{Q}$-factorial variety over an algebraically closed field of characteristic zero. We denote by $N^1(X)$ the real vector space of $\mathbb{R}$-Cartier divisors modulo numerical equivalence. 
The \emph{nef cone} of $X$ is the closed convex cone $\Nef(X)\subset N^1(X)$ generated by classes of nef divisors. 

The stable base locus $\textbf{B}(D)$ of a $\mathbb{Q}$-divisor $D$ is the set-theoretic intersection of the base loci of the complete linear systems $|sD|$ for all positive integers $s$ such that $sD$ is integral
\stepcounter{thm}
\begin{equation}\label{sbl}
\textbf{B}(D) = \bigcap_{s > 0}B(sD)
\end{equation}
The \emph{movable cone} of $X$ is the convex cone $\Mov(X)\subset N^1(X)$ generated by classes of 
\emph{movable divisors}. These are Cartier divisors whose stable base locus has codimension at least two in $X$.
The \emph{effective cone} of $X$ is the convex cone $\Eff(X)\subset N^1(X)$ generated by classes of 
\emph{effective divisors}. We have inclusions $\Nef(X)\ \subset \ \overline{\Mov(X)}\ \subset \ \overline{\Eff(X)}$. We refer to \cite[Chapter 1]{De01} for a comprehensive treatment of these topics. 

\begin{Definition}
A \textit{spherical variety} is a normal variety $X$ together with an action of a connected reductive affine algebraic group $\mathscr{G}$, a Borel subgroup $\mathscr{B}\subset \mathscr{G}$, and a base point $x_0\in X$ such that the $\mathscr{B}$-orbit of $x_0$ in $X$ is a dense open subset of $X$. 

Let $(X,\mathscr{G},\mathscr{B},x_0)$ be a spherical variety. We distinguish two types of $\mathscr{B}$-invariant prime divisors: a \textit{boundary divisor} of $X$ is a $\mathscr{G}$-invariant prime divisor on $X$, a \textit{color} of $X$ is a $\mathscr{B}$-invariant prime divisor that is not $\mathscr{G}$-invariant. We will denote by $\mathcal{B}(X)$ and $\mathcal{C}(X)$ respectively the set of boundary divisors and colors of $X$.
\end{Definition}

For instance, any toric variety is a spherical variety with $\mathscr{B}=\mathscr{G}$ equal to the torus. For a toric variety there are no colors, and the boundary divisors are the usual toric invariant divisors. In the following, we will carefully analyze the natural action of $SL(n+1)$ on $\mathcal{A}(n)$ in order to get information on the cones of divisors of this space.

\begin{Remark}\label{l1}
The $SL(n+1)$-action
$$
\begin{array}{cccc}
SL(n+1)\times \mathcal{G}(1,n) & \longrightarrow & \mathcal{G}(1,n)\\
(A,[v\wedge w]) & \longmapsto & [Av\wedge Aw]
\end{array}
$$
is the restriction to $\mathcal{G}(1,n)$ of the $SL(n+1)$-action on $\mathbb{P}^{N_{-}}$ given by
$$
\begin{array}{cccc}
SL(n+1)\times \mathbb{P}^{N_{-}} & \longrightarrow & \mathbb{P}^{N_{-}}\\
(A,Z) & \longmapsto & AZA^{t}
\end{array}
$$
\end{Remark}
 
\begin{Definition}
A \textit{wonderful variety} is a smooth projective variety $X$ with an action of a semi-simple simply connected group $\mathscr{G}$ such that:
\begin{itemize}
\item[-] there is a point $x_0\in X$ with open $\mathscr{G}$ orbit and such that the complement $X\setminus \mathscr{G}\cdot x_0$ is a union of prime divisors $E_1,\cdots, E_r$ having simple normal crossing;
\item[-] the closures of the $\mathscr{G}$-orbits in $X$ are the intersections $\bigcap_{i\in I}E_i$ where $I$ is a subset of $\{1,\dots, r\}$.
\end{itemize} 
\end{Definition} 
 
\begin{Remark}\label{sw}
In \cite{Lu96} D. Luna proved that a wonderful variety over an algebraically closed field of characteristic zero with an action of a semi-simple simply connected group $\mathscr{G}$ is spherical, meaning that it is almost homogeneous under any Borel subgroup $\mathscr{B}$ of $\mathscr{G}$.
\end{Remark}
 
\begin{Proposition}\label{p1}
The variety $\mathcal{A}(n)$ is wonderful, and hence spherical. The Picard group of $\mathcal{A}(n)$ is given by 
$$\Pic(\mathcal{A}(n)) = 
\left\lbrace\begin{array}{ll}
\mathbb{Z}[H^{-},E_1^{-},\dots,E_{\frac{n-3}{2}}^{-}] & \text{if n is odd}\\ 
\mathbb{Z}[H^{-},E_1^{-},\dots,E_{\frac{n-2}{2}}^{-}] & \text{if n is even}
\end{array}\right.
$$
and the set of boundary divisors and colors of $\mathcal{A}(n)$ are given respectively by
$$\mathcal{B}(\mathcal{A}(n)) = 
\left\lbrace\begin{array}{ll}
\{E_1^{-},\dots,E_{\frac{n-1}{2}}^{-}\} & \text{if n is odd}\\ 
\{E_1^{-},\dots,E_{\frac{n-2}{2}}^{-}\} & \text{if n is even}
\end{array}\right.
$$
$$
\mathcal{C}(\mathcal{A}(n)) = 
\left\lbrace\begin{array}{ll}
\{D_{2j+2}^{-},\: j = 0,\dots,\frac{n-3}{2}\} & \text{if n is odd}\\ 
\{D_{2j+2}^{-},\: j = 0,\dots,\frac{n-2}{2}\} & \text{if n is even}
\end{array}\right.
$$
where the $D_{2j+2}^{-}$ are the divisors defined in Theorem \ref{main}.
\end{Proposition}
\begin{proof}
By Construction \ref{css} the space of complete skew-forms can be obtained as a sequence of blow-ups of smooth varieties along smooth centers. Note that if $n$ is odd then the strict transform of $\sec_{\frac{n-1}{2}}(\mathcal{G}(1,n))$ is a Cartier divisor in $\mathcal{A}[n]_{n-1}$ and hence the last blow-up in Construction \ref{css} is an isomorphism.

By \cite[Theorem 11.1]{Tha99} the variety $\mathcal{A}(n)$ is wonderful with the action of $\mathscr{G}=SL(n+1)$ in Remark \ref{l1}. Therefore, by Remark \ref{sw} it is in particular spherical. We will consider the Borel subgroup
\begin{equation}\label{borel}
\mathscr{B} = \{A\in SL(n+1) \: |\: A \: \text{is upper triangular}\}
\end{equation}

By Remark \ref{l1} an element $A\in \mathscr{G}$ maps the matrix $Z$ to the matrix $\overline{Z}=AZA^{t}$. Now, set
$$
A_{2,2}^i=\left(
\begin{array}{ccc}
a_{n-i+1,n-i+1} & \dots & a_{n-i+1,n}\\ 
\vdots & \ddots & \vdots\\ 
a_{n,n-i+1} & \dots & a_{n,n}
\end{array}\right)
$$
$$
Z_{2,2}^i=\left(
\begin{array}{ccc}
z_{n-i+1,n-i+1} & \dots & z_{n-i+1,n}\\ 
\vdots & \ddots & \vdots\\ 
z_{n,n-i+1} & \dots & z_{n,n}
\end{array}\right)\quad
\overline{Z}_{2,2}^i=\left(
\begin{array}{ccc}
\overline{z}_{n-i+1,n-i+1} & \dots & \overline{z}_{n-i+1,n}\\ 
\vdots & \ddots & \vdots\\ 
\overline{z}_{n,n-i+1} & \dots & \overline{z}_{n,n}
\end{array}\right)
$$
and subdivide the matrices $A,Z,\overline{Z}$ in blocks as follows
$$
A = \left(\begin{array}{cc}
A_{1,1}^i & A_{1,2}^i\\ 
A_{2,1}^i & A_{2,2}^i
\end{array}\right)\:
Z = \left(\begin{array}{cc}
Z_{1,1}^i & Z_{1,2}^i\\ 
Z_{2,1}^i & Z_{2,2}^i
\end{array}\right)\:
\overline{Z} = \left(\begin{array}{cc}
\overline{Z}_{1,1}^i & \overline{Z}_{1,2}^i\\ 
\overline{Z}_{2,1}^i & \overline{Z}_{2,2}^i
\end{array}\right)
$$
Then for $i=1,\dots, n$ the matrix $AZA^{t}$ can be subdivided in four blocks as follows
$$
\left\lbrace\begin{array}{l}
\overline{Z}_{1,1}^i = A_{1,1}^iZ_{1,1}^iA_{1,1}^{i\:t}+A_{1,2}^iZ_{2,1}^iA_{1,1}^{i\:t}+A_{1,1}^{i\:t}Z_{1,2}^{i\:t}A_{1,2}^{i\:t}+A_{1,2}^{i\:t}Z_{2,2}^{i\:t}A_{1,2}^{i\:t}\\ 
\overline{Z}_{1,2}^i = A_{1,1}^iZ_{1,1}^iA_{2,1}^{i\:t}+A_{1,2}^iZ_{2,1}^iA_{2,1}^{i\:t}+A_{1,1}^iZ_{1,2}^iA_{2,2}^{i\:t}+A_{1,2}^iZ_{2,2}^iA_{2,2}^{i\:t}\\
\overline{Z}_{2,1}^i = A_{2,1}^iZ_{1,1}^iA_{1,1}^{i\:t}+A_{2,2}^iZ_{2,1}^iA_{1,1}^{i\:t}+A_{2,1}^iZ_{1,2}^iA_{1,2}^{i\:t}+A_{2,2}^iZ_{2,2}^iA_{1,2}^{i\:t}\\ 
\overline{Z}_{2,2}^i = A_{2,1}^iZ_{1,1}^iA_{2,1}^{i\:t}+A_{2,2}^iZ_{2,1}^iA_{2,1}^{i\:t}+A_{2,1}^iZ_{1,2}^iA_{2,2}^{i\:t}+A_{2,2}^iZ_{2,2}^iA_{2,2}^{i\:t}
\end{array}\right.
$$
Let us focus on the case $n$ odd, the even case can be worked out similarly. The degree $i$ hypersurface $\{\pf(Z_{2,2}^i)=0\}\subset\mathbb{P}^{N_{-}}$ is stabilized by the action of $\mathscr{G}$ if and only if $i = n+1$. Indeed, $\{\pf(Z_{2,2}^{n+1})=0\}\subset\mathbb{P}^{N_{-}}$ is the last secant variety of $\mathcal{G}(1,n)$ not filling the whole of $\mathbb{P}^{N_{-}}$. Note that since $\mathscr{G}$ stabilizes $\mathcal{G}(1,n)$ it must stabilize all its secant varieties. Now, consider the Borel subgroup $\mathscr{B}$ in (\ref{borel}) and take a matrix $A\in\mathscr{B}$. Then $A_{2,1}^i$ is the zero matrix, and $\overline{Z}_{2,2}^i = A_{2,2}^iZ_{2,2}^iA_{2,2}^{i\:t}$. Since we are in the skew-symmetric case $\overline{Z}_{2,2}^i$, $Z_{2,2}^i$ are skew-symmetric. Then we get 
$$\pf(\overline{Z}_{2,2}^i) = \pf(A_{2,2}^iZ_{2,2}^iA_{2,2}^{i\:t}) = \det(A_{2,2}^i)\pf(Z_{2,2}^i)$$
Note that $A_{2,2}^i$ is upper triangular while $A_{2,2}^{i\:t}$ is lower triangular, and that the diagonals of $A_{2,2}^i$ and $A_{2,2}^{i\:t}$ are made of elements of the diagonal of $A$. 

Therefore, $\det(A_{2,2}^i)\neq 0$, $\det(A_{2,2}^{i\:t})\neq 0$ and hence the hypersurface $\{\pf(Z_{2,2}^i)=0\}\subset\mathbb{P}^{N_{-}}$ is stabilized by the action of the Borel subgroup $\mathscr{B}$ on $\mathbb{P}^{N_{-}}$. Therefore, the strict transform $D_{2i+2}$ of $\{\pf(Z_{2,2}^i)=0\}\subset \mathbb{P}^{N_{-}}$ is stabilized by the action of $\mathscr{B}$, and it is stabilized by the action of $\mathscr{G}$ on $\mathcal{A}(n)$ if and only if $i = n+1$. 

As noticed in \cite[Remark 4.5.5.3]{ADHL15}, if $(X,\mathscr{G},\mathscr{B},x_0)$ is a spherical wonderful variety with colors $D_1,\dots,D_s$ the big cell $X\setminus (D_1\cup\dots \cup D_s)$ is an affine space. Therefore, it admits only constant invertible global functions and $\Pic(X) = \mathbb{Z}[D_1,\dots,D_s]$. Now, the Picard rank of $\mathcal{A}(n)$ is $\frac{n-3}{2}+1$ if $n$ is odd and $\frac{n-2}{2}+1$ if $n$ is even, and we found exactly $\frac{n-3}{2}+1$ colors if $n$ is odd and $\frac{n-2}{2}+1$ colors if $n$ is even. This concludes the proof of the statement on the set of colors of $\mathcal{A}(n)$.

Now, note that since $\mathscr{G}$ stabilizes the secant varieties of $\mathcal{G}(1,n)$, \cite[Chapter II, Section 7, Corollary 7.15]{Har77} yields that $\mathscr{G}$ stabilizes the exceptional divisors $E_i^{-}$ and the strict transform of the last secant variety of $\mathcal{G}(1,n)$. Therefore, these are boundary divisors. 

Now, let $D\subset\mathcal{A}(n)$ be a $\mathscr{G}$-invariant divisor which is not exceptional for the blow-up morphism $f^{-}:\mathcal{A}(n)\rightarrow \mathbb{P}^{N_{-}}$ in Construction \ref{css}. Then $f_{*}^{-}D\subset\mathbb{P}^{N_{-}}$ is $\mathscr{G}$-invariant as well. Hence, in particular $f_{*}^{-}D\subset\mathbb{P}^{N_{-}}$ is $\mathscr{B}$-invariant. On the other hand, by the previous computation of the colors of $\mathcal{A}(n)$ we have that the only $\mathscr{G}$-invariant hypersurface in $\mathbb{P}^{N_{-}}$ is the last secant variety of $\mathcal{G}(1,n)$. 
\end{proof}

The following result will be fundamental in order to write down the classes of the colors in Proposition \ref{p1} in terms of the generators of the Picard group.

\begin{Lemma}\label{mult}
Consider the hypersurface $Y_k^{-}\subset\mathbb{P}^{N_{-}}$ defined as the zero locus of a $(2k+2)\times (2k+2)$ sub-Pfaffian of the matrix $Z$ in (\ref{matrix}). Then 
$$
\mult_{\sec_h(\mathcal{G}(1,n))}Y_k^{-} =
\left\lbrace\begin{array}{ll}
k-h+1 & \text{if h $\leq$ k}\\ 
0 & \text{if h $>$ k}
\end{array}\right.
$$ 
\end{Lemma}
\begin{proof}
If $h > k$ then $\deg(Y_k^{-}) = k+1< h+1$ and since $I(\sec_h(\mathcal{G}(1,n)))$ is generated by the $(2h+2)\times (2h+2)$ sub-Pfaffians, which are hypersurfaces of degree $h+1$, the hypersurface $Y_k^{-}$ can not contain $\sec_h(\mathcal{G}(1,n))$.

Let $h\leq k$. Without loss of generality we may consider the hypersurface $Y_k^{-}$ given by $Y_k = \{\pf(G)=0\}$ where 
\stepcounter{thm}
\begin{equation}\label{pol1}
G = F(z_{0,1},\dots,z_{2k,2k+1}) = \det\left(
\begin{array}{cccc}
0 & z_{0,1} & \dots & z_{0,2k+1}\\ 
-z_{0,1} & \dots & \dots & z_{1,2k+1}\\ 
\vdots & \ddots & \ddots & \vdots\\
-z_{0,2k} & \dots & \dots & z_{2k,2k+1}\\
-z_{0,2k+1} & \dots & -z_{2k,2k+1} & 0 
\end{array}\right)
\end{equation}   
Note that
$$\frac{\partial^j \pf(G)}{\partial z_{0,0}^{j_{0,0}},\dots,\partial z_{k,k}^{j_{k,k}}}=0, \: j_{0,0}+\dots +j_{k,k}$$
whenever either $j_{r,s}\geq 2$ for some $r,s = 0,\dots,k$ or in the expression of the partial derivative there are at least two indexes either of type $j_{r,s},j_{r',s}$ or of type $j_{r,s},j_{s,s'}$. In all the other cases a partial derivative of order $j$ is the Pfaffian of the $(2k+2-2j)\times (2k+2-2j)$ minor of the matrix in \ref{pol1} obtained by deleting the rows and the columns crossing in the elements which correspond to the variables with respect to which we are deriving. Note that for each variable we are deleting two rows and two columns. 

Therefore, all the partial derivatives of order $j$ of $\pf(G)$ vanish on $\sec_{k-j}(\mathcal{G}(1,n))$. On the other hand, since the ideal of $\sec_{k-j}(\mathcal{G}(1,n))$ is generated in degree $k-j+1$, the non-zero partial derivatives of order $j+1$ of $\pf(G)$, which have degree $k-j$, can not belong to $I(\sec_{k-j}(\mathcal{G}(1,n)))$. Therefore, $\mult_{\sec_{k-j}(\mathcal{G}(1,n))}Y_k^{-}=j+1$ that is $\mult_{\sec_h(\mathcal{G}(1,n))}Y_k^{-} = h-k+1$ if $h\leq k$.   
\end{proof}

Now, we are ready to compute the effective and nef cones of the spaces of complete skew-forms.

\begin{thm}\label{theff}
If $n$ is odd we have $D_{2}^{-}\sim H^{-}$,
$$D_{2k+2}^{-}\sim (k+1)H^{-}-\sum_{h=1}^k(k-h+1)E_{h}^{-}$$
for $k = 1,\dots,\frac{n-3}{2}$, and $D_{n+1}^{-}\sim\frac{n+1}{2}H^{-}-\sum_{h=1}^{\frac{n-3}{2}}\frac{n-2h+1}{2}E_h$.

Furthermore, $\{E_{1}^{-},\dots,E_{\frac{n-3}{2}}^{-},D_{n+1}^{-}\}$ generate the extremal rays of $\Eff(\mathcal{A}(n))$ and $\{D_{2j+2}^{-},\: j = 0,\dots,\frac{n-3}{2}\}$ generate the extremal rays of $\Nef(\mathcal{A}(n))$.

If $n$ is even we have $D_{2}^{-}\sim H^{-}$,
$$D_{2k+2}^{-}\sim (k+1)H^{-}-\sum_{h=1}^k(k-h+1)E_{h}^{-}$$
for $k = 1,\dots,\frac{n-2}{2}$, $\{E_1^{-},\dots,E_{\frac{n-2}{2}}^{-},D_{n}^{-}\}$ generate the extremal rays of $\Eff(\mathcal{A}(n))$ and moreover $\{D_{2j+2}^{-},\: j = 0,\dots,\frac{n-2}{2}\}$ generate the extremal rays of $\Nef(\mathcal{A}(n))$.
\end{thm}
\begin{proof}
Let $Y$ be a smooth and irreducible subvariety of a smooth variety $X$, and let $f:Bl_YX\rightarrow X$ be the blow-up of $X$ along $Y$ with exceptional divisor $E$. Then for any divisor $D\in \Pic(X)$ in $\Pic(Bl_YX)$ we have
$$\widetilde{D} \sim f^{*}D-\mult_{Y}(D) E$$
where $\widetilde{D}\subset Bl_YX$ is the strict transform of $D$, and $\mult_Y(D)$ is the multiplicity of $D$ at a general point of $Y$. 

This observation, together with Construction \ref{css} and Lemma \ref{mult}, leads us to the expression for $D_{2k+2}$ in the statement. Just notice that when $n$ is odd $\sec_{\frac{n-1}{2}}(\mathcal{G}(1,n))\subset\mathbb{P}^{N_{-}}$ is a hypersurface of degree $\frac{n+1}{2}$, and its strict transform via the morphism $f^{-}:\mathcal{A}(n)\rightarrow\mathbb{P}^{N_{-}}$ in Construction \ref{css} is the divisor $E^{-}_{\frac{n-1}{2}}$. The statements on the effective and the nef cones follow respectively from \cite[Proposition 4.5.4.4]{ADHL15} and \cite[Section 2.6]{Br89}.
\end{proof}

In the following, we will denote by $l^{-},e_i^{-}$ the classes in $N_1(\mathcal{A}(n))$ such that $H^{-}\cdot l^{-} = 1, H^{-}\cdot e_i^{-} = 0$ and $E^{-}_i\cdot l^{-} = 0, E^{-}_i\cdot e^{-}_j = -\delta_{i,j}$.

\begin{Proposition}\label{Morimov}
Let us write the class of a general curve $C\in N_1(\mathcal{A}(n))$ as $C = dl^{-}-\sum_{i}m_ie_i^{-}$. Then the cone of moving curves $\mov(\mathcal{A}(n))$ is defined by 
$$
\left\lbrace
\begin{array}{l}
m_i\geq 0 \text{ for $i=1,\dots \frac{n-3}{2}$}\\ 
\frac{n+1}{2}d-\sum_{i=1}^{\frac{n-3}{2}}\frac{n-2i+1}{2}m_i\geq 0
\end{array}\right. \text{if n is odd}
\quad
\left\lbrace\begin{array}{l}
m_i\geq 0 \text{ for $i=1,\dots \frac{n-2}{2}$}\\ 
\frac{n}{2}d-\sum_{i=1}^{\frac{n-2}{2}}\frac{n-2i}{2}m_i\geq 0 
\end{array}\right. \text{if n is even}
$$
Furthermore, the extremal rays of the Mori cone of $\mathcal{A}(n)$ are generated by the curves of class $l^{-}-2e_1^{-}+e_2^{-}$, $e_i^{-}-2e_{i+1}^{-}+e_{i+2}^{-}$ for $i = 1,\dots,\frac{n-7}{2}$, $e_{\frac{n-5}{2}}^{-}-2e_{\frac{n-3}{2}}^{-}$ and $e_{\frac{n-3}{2}}^{-}$ if $n$ is odd, and by the curves of class  $l^{-}-2e_1^{-}+e_2^{-}$, $e_i^{-}-2e_{i+1}^{-}+e_{i+2}^{-}$ for $i = 1,\dots,\frac{n-6}{2}$, $e_{\frac{n-4}{2}}^{-}-2e_{\frac{n-2}{2}}^{-}$ and $e_{\frac{n-2}{2}}^{-}$ if $n$ is even.
\end{Proposition}
\begin{proof}
The Mori cone is dual to the nef cone, and by \cite[Theorem 2.2]{BDPP13} the cone of moving curves is dual to the effective cone. Therefore, to get the statement it is enough to apply Theorem \ref{theff}.
\end{proof}

\begin{Corollary}\label{Fano}
The space of complete skew-forms $\mathcal{A}(n)$ is a Fano variety of index $i_{\mathcal{A}(n)}$ given by the following table
\begin{center}
\begin{tabular}{c|c}
$n$ & $i_{\mathcal{A}(n)}$ \\ 
\hline 
$2$           & $3$   \\ 
$3$           & $6$   \\ 
$4$           & $2$   \\ 
$5$           & $5$   \\  
$\geq 6$    & $1$ \\ 
\hline 
\end{tabular} 
\end{center} 
\end{Corollary}
\begin{proof}
We can write down $-K_{\mathcal{A}(n)}$ on the basis of $\Pic(\mathcal{A}(n))$ given by $H^{-}$ and the $E_i^{-}$. By Remark \ref{dimsec} we have
$$\codim_{\P^{N_{-}}}(\sec_{h}(\mathcal{G}(1,n))) = \frac{n^2+n-2h(2n-2h+1)}{2}$$ 
for any $h\leq \lfloor \frac{n-3}{2}\rfloor +1$. A straightforward computation shows that $-K_{\mathcal{A}(n)}$ has positive intersection with all the generators of $\NE(\mathcal{A}(n))$ in Proposition \ref{Morimov}.

Now, let us compute the index. First note that $\mathcal{A}(2)\cong\mathbb{P}^2$ and $\mathcal{A}(3)\cong\mathbb{P}^5$. Furthermore, we have $-K_{\mathcal{A}(4)} = 10H^{-}-2E_1^{-}$ and $-K_{\mathcal{A}(3)}\cdot e_1^{-} = 2$. Similarly, $-K_{\mathcal{A}(5)} = 15H^{-}-5E_1^{-}$ and $-K_{\mathcal{A}(3)}\cdot e_1^{-} = 5$. Finally, when $n\geq 6$ we distinguish two cases. If $n$ is odd the discrepancy of the exceptional divisor over the last blown-up secant variety, that is for $h = \frac{n-3}{2}$, is $5$ and the discrepancy of the exceptional divisor over the second last blown-up secant variety, that is for $h = \frac{n-3}{2}-1$, is $14$. Since $5$ and $14$ are coprime, by intersecting $-K_{\mathcal{A}(n)}$ with $e_{\frac{n-3}{2}}^{-}$ and $e_{\frac{n-3}{2}-1}^{-}$ we get that $i_{\mathcal{A}(n)} = 1$. Similarly, when $n$ is even the discrepancy of the exceptional divisor over the last blown-up secant variety, that is for $h = \frac{n-2}{2}$, is $2$ and the discrepancy of the exceptional divisor over the second last blown-up secant variety, that is for $h = \frac{n-2}{2}-1$, is $9$.
\end{proof}

\section{On the Cox ring}\label{sec3}
Let $X$ be a normal $\mathbb{Q}$-factorial variety with free and finitely generated divisor class group $\Cl(X)$. Fix a subgroup $G$ of the group of Weil divisors on $X$ such that the canonical map $G\rightarrow\Cl(X)$, mapping a divisor $D\in G$ to its class $[D]$, is an isomorphism. The \textit{Cox ring} of $X$ is defined as
$$\Cox(X) = \bigoplus_{[D]\in \Cl(X)}H^0(X,\mathcal{O}_X(D))$$
where $D\in G$ represents $[D]\in\Cl(X)$, and the multiplication in $\Cox(X)$ is defined by the standard multiplication of homogeneous sections in the field of rational functions on $X$. If $\Cox(X)$ is finitely generated as an algebra over the base field, then $X$ is said to be a \textit{Mori dream space}. A perhaps more enlightening definition, especially for the relation with the minimal model program, is the following. 

\begin{Definition}\label{def:MDS} 
A normal projective $\mathbb{Q}$-factorial variety $X$ is called a \emph{Mori dream space}
if the following conditions hold:
\begin{enumerate}
\item[-] $\Pic{(X)}$ is finitely generated, or equivalently $h^1(X,\mathcal{O}_X)=0$,
\item[-] $\Nef{(X)}$ is generated by the classes of finitely many semi-ample divisors,
\item[-] there is a finite collection of small $\mathbb{Q}$-factorial modifications
 $f_i: X \dasharrow X_i$, such that each $X_i$ satisfies the second condition above, and $
 \Mov{(X)} \ = \ \bigcup_i \  f_i^*(\Nef{(X_i)})$.
\end{enumerate}
\end{Definition}

The collection of all faces of all cones $f_i^*(\Nef{(X_i)})$ above forms a fan which is supported on $\Mov(X)$.
If two maximal cones of this fan, say $f_i^*(\Nef{(X_i)})$ and $f_j^*(\Nef{(X_j)})$, meet along a facet,
then there exist a normal projective variety $Y$, a small modification $\varphi:X_i\dasharrow X_j$, and $h_i:X_i\rightarrow Y$, $h_j:X_j\rightarrow Y$ small birational morphisms of relative Picard number one such that $h_j\circ\varphi = h_i$. The fan structure on $\Mov(X)$ can be extended to a fan supported on $\Eff(X)$ as follows. 

\begin{Definition}\label{MCD}
Let $X$ be a Mori dream space.
We describe a fan structure on the effective cone $\Eff(X)$, called the \emph{Mori chamber decomposition}.
We refer to \cite[Proposition 1.11]{HK00} and \cite[Section 2.2]{Ok16} for details.
There are finitely many birational contractions from $X$ to Mori dream spaces, denoted by $g_i:X\dasharrow Y_i$.
The set $\Exc(g_i)$ of exceptional prime divisors of $g_i$ has cardinality $\rho(X/Y_i)=\rho(X)-\rho(Y_i)$.
The maximal cones $\mathcal{C}$ of the Mori chamber decomposition of $\Eff(X)$ are of the form: $\mathcal{C}_i \ = \left\langle g_i^*\big(\Nef(Y_i)\big) , \Exc(g_i) \right\rangle$. We call $\mathcal{C}_i$ or its interior $\mathcal{C}_i^{^\circ}$ a \emph{maximal chamber} of $\Eff(X)$.
\end{Definition}

\begin{Remark}\label{sphMDS}
By the work of M. Brion \cite{Br93} we have that $\mathbb{Q}$-factorial spherical varieties are Mori dream spaces. An alternative proof of this result can be found in \cite[Section 4]{Pe14}. In particular, by Proposition \ref{p1} $\mathcal{A}(n)$ is a Mori dream space.
\end{Remark}

\begin{Remark}\label{dimCox}
By \cite[Proposition 2.9]{HK00} a normal and $\mathbb{Q}$-factorial projective
variety $X$ over an algebraically closed field $K$, with finitely generated Picard group is a Mori dream space if and only if $\Cox(X)$ is a finitely generated $K$-algebra. Furthermore, the following equality holds:
$$\dim\Cox(X) = \dim(X)+\rank\Pic(X)$$
where $\dim\Cox(X)$ is the Krull dimension of $\Cox(X)$, see for instance \cite[Theorem 3.2.1.4]{ADHL15}.
\end{Remark}

As a consequence of the computation of colors and boundary divisors in Proposition \ref{p1} we have the following result which will be central in the rest of the paper.

\begin{thm}\label{gen}
Let $Z_I^{-}$ be the minor of the matrix $Z$ (\ref{matrix}) built with the rows and columns indexed by $I = \{i_0,\dots,i_{2k+1}\}$, $T_{I}^{-}$ the canonical section associated to the strict transform of the hypersurface $\{\pf(Z_{I}^{-})=0\}\subset\mathbb{P}^{N_{-}}$, and $S_i^{-}$ the canonical section associated to the exceptional divisor $E_i^{-}$ in Construction \ref{css}. 

Then $\Cox(\mathcal{A}(n))$ is generated by the $T_{I}^{-}$ and the $S_i^{-}$ with $2\leq |I|\leq n-1$, $1\leq i\leq\frac{n-1}{2}$ if $n$ is odd, and $2\leq |I|\leq n$, $1\leq i\leq\frac{n-2}{2}$ if $n$ is even.
\end{thm}
\begin{proof}
By \cite[Theorem 4.5.4.6]{ADHL15} if $\mathscr{G}$ is a semi-simple and simply connected algebraic group and $(X,\mathscr{G},\mathscr{B},x_0)$ is a spherical variety with boundary divisors $E_1,\dots,E_r$ and colors $D_1,\dots,D_s$ then $\Cox(X)$ is generated as a $K$-algebra by the canonical sections of the $E_i$ and the finite dimensional vector subspaces $\lin_{K}(\mathscr{G}\cdot D_i)\subseteq \Cox(X)$ for $1\leq i\leq s$, where $\lin_{K}(\mathscr{G}\cdot D_i)$ is the linear subspace of $H^0(\mathcal{A}(n),\mathcal{O}_{\mathcal{A}(n)}(D_i))$ spanned by the orbit $\mathscr{G}\cdot D_i$, that is the smallest linear subspace of $H^0(\mathcal{A}(n),\mathcal{O}_{\mathcal{A}(n)}(D_i))$ containing the orbit $\mathscr{G}\cdot D_i$.

The colors of $\mathcal{A}(n)$ have been computed in Proposition \ref{p1}. Let $Z_{2k+2}^{br-}$ be the $(2k+2)\times (2k+2)$ bottom right minor of the matrix $Z$ in (\ref{matrix}). The Pl\"ucker embedding $\mathcal{G}(1,n)\hookrightarrow\mathbb{P}(\bigwedge^2V)$ maps a point    
$$
\left(
\begin{array}{ccc}
x_0 & \dots & x_n\\ 
y_0 & \dots & y_n
\end{array} 
\right)\in \mathcal{G}(1,n)
$$ 
to a skew-symmetric matrix with entries $z_{i,j} = x_iy_j-x_jy_i$. Let $Z_I^{-}$ be the $(2k+2)\times (2k+2)$ minor of $Z$ corresponding to $I = \{i_0,\dots,i_{2k+1}\}$, and consider an element $a\in SL(n+1)$ such that $a\cdot x_{n-2k-1} = x_{i_0},\dots, a\cdot x_n = x_{i_{2k+1}}$ and $a\cdot y_{n-2k-1} = y_{i_0},\dots, a\cdot y_n = y_{i_{2k+1}}$. Then $a\cdot \pf(Z_{2k+2}^{br-})=\pf(Z_{I}^{-})$. Finally, since $I(\sec_{k}(\mathcal{G}(1,n)))$ is generated by $(2k+2)\times (2k+2)$ sub-Pfaffians of $Z^{-}$ we conclude that $\lin_{K}(\mathscr{G}\cdot \pf(Z_{I}^{br-})) = I(\sec_{k}(\mathcal{G}(1,n)))$.
\end{proof}

As a first application of Theorem \ref{gen} we compute the number of extremal rays of the movable cones of the spaces of complete skew-forms. 

\begin{Proposition}\label{PropMov}
The movable cone of $\mathcal{A}(n)$ is generated by $2^{k-1}$ extremal rays if $n = 2k+1$ is odd, and by $2^{k-2}+1$ extremal rays if $n = 2k$ is even. 
\end{Proposition}
\begin{proof}
It is enough to apply \cite[Proposition 3.3.2.3]{ADHL15} to the generators of the Cox ring in Theorem \ref{gen}.
\end{proof}

\begin{Remark}
As an ancillary file in the arXiv version of the paper we include the Maple script ${\tt MovableCSS}$, based on Theorem \ref{gen} and \cite[Proposition 3.3.2.3]{ADHL15}, computing the extremal rays of $\Mov(\mathcal{A}(n))$.

For instance, with respect to the standard basis of the Picard group given by $H^{-}$ and the $E_{i}^{-}$, $\Mov(\mathcal{A}(6))$ is generated by $(3,-2,-1)$, $(1,0,0)$, $(2,-1,0)$, $\Mov(\mathcal{A}(7))$ is generated by $(3,-2,-1)$, $(1, 0, 0)$, $(6, -3, -2)$, $(2,-1,0)$, and $\Mov(\mathcal{A}(8))$ is generated by $(4,-3,-2,-1)$, $(3,-2,-1,0)$, $(1,0,0,0)$, $(2,-1,0,0)$, $(6,-3,-2,0)$.
\end{Remark}

Finally, by the following result we get an explicit description of the relations among the generators of $\Cox(\mathcal{A}(n))$ when $n\in \{4,5\}$.

\begin{thm}
The Cox rings of $\mathcal{A}(4)$ and $\mathcal{A}(5)$ are given respectively by 
$$\Cox(\mathcal{A}(4))\cong\frac{K[T_I^{-},S_1]_{ |I|\in\{2,4\}}}{I(\mathbb{S}_5)},\quad \Cox(\mathcal{A}(5))\cong\frac{K[T_I^{-},S_1,S_2]_{ |I|\in\{2,4\}}}{I(\mathbb{S}_6)}$$
where $\mathbb{S}_{n+1}\subset\mathbb{P}^{2^n-1}$, for $n\in\{4,5\}$, is the $\frac{n(n+1)}{2}$-dimensional Spinor variety.
\end{thm}
\begin{proof}
Note that by Theorem \ref{gen} $\Cox(\mathcal{A}(4))$ and $\Cox(\mathcal{A}(5))$ have respectively $16$ and $32$ generators. For simplicity of notation we will develop in full detail the case $n = 4$.

Consider the matrix $Z$ in (\ref{matrix}). The Spinor variety $\mathbb{S}_{5}$ is the closure of the image of the map
$$
\begin{array}{ccc}
\bigwedge^2V& \rightarrow & \mathbb{P}^{15}\\
 Z & \longmapsto & (1,z_{0,1},\dots,z_{3,4},\pf(Z_0),\dots,\pf(Z_4))
\end{array}
$$
where $Z_i$ is the minor obtained by deleting the $i$-th row and column of $Z$. When we take the closure we add the missing variable $S_1$, and the $10$ quadrics cutting out $\mathbb{S}_{5}\subset\mathbb{P}^{15}$ induce $10$ quadratic relations among the generators of $\Cox(\mathcal{A}(4))$. Now, Remark \ref{dimCox} yields $\dim \Cox(\mathcal{A}(4)) = 9 + 2 = 11$, and to conclude it is enough to observe that $\mathbb{S}_{5}\subset\mathbb{P}^{15}$ is a $10$-dimensional irreducible and reduced variety.

In the case $n = 5$ we may argue in a completely analogous way observing that $\dim\Cox(\mathcal{A}(5) = 14 + 2 = 16 = \dim(\mathbb{S}_6)+1$. 
\end{proof}

\section{On the Mori chamber and stable base locus decompositions}\label{sec4}
In this section we will study the Mori chamber and stable base locus decompositions for spaces of complete skew-forms. The stable base locus of an effective $\mathbb{Q}$-divisor on a normal $\mathbb{Q}$-factorial projective variety $X$ has been defined in (\ref{sbl}). Since stable base loci do not behave well with respect to numerical equivalence \cite[Example 10.3.3]{La04II}, we will assume that $h^{1}(X,\mathcal{O}_X)=0$ so that linear and numerical equivalence of $\mathbb{Q}$-divisors coincide. Then numerically equivalent $\mathbb{Q}$-divisors on $X$ have the same stable base locus, and the pseudo-effective cone $\overline{\Eff}(X)$ of $X$ can be decomposed into chambers depending on the stable base locus of the corresponding linear series. This decomposition
is called \textit{stable base locus decomposition}. This means that $\overline{\Eff}(X$ can be decomposed into regions given by unions of possibly non convex cones such that $D_1,D_2\in\overline{\Eff}(X$ belongs to the same region if and only if $\textbf{B}(D_1) = \textbf{B}(D_2)$, see \cite[Section 4.1.3]{CdFG17} for further details. Note that if $X$ is a Mori dream space then $h^1(X,\mathcal{O}_X)=0$.

Recall that two divisors $D_1,D_2$ on a Mori dream space belong to the same Mori chamber if and only if $\textbf{B}(D_1) = \textbf{B}(D_2)$ and the following diagram of rational maps is commutative
   \[
  \begin{tikzpicture}[xscale=1.5,yscale=-1.2]
    \node (A0_1) at (1, 0) {$X$};
    \node (A1_0) at (0, 1) {$X(D_1)$};
    \node (A1_2) at (2, 1) {$X(D_2)$};
    \path (A1_0) edge [->]node [auto] {$\scriptstyle{}$} node [rotate=180,sloped] {$\scriptstyle{\widetilde{\ \ \ }}$} (A1_2);
    \path (A0_1) edge [->,dashed]node [auto] {$\scriptstyle{\phi_{D_2}}$} (A1_2);
    \path (A0_1) edge [->,swap, dashed]node [auto] {$\scriptstyle{\phi_{D_1}}$} (A1_0);
  \end{tikzpicture}
  \]
where the horizontal arrow is an isomorphism. Therefore, the Mori chamber decomposition is a refinement of the stable base locus decomposition.  

\begin{Remark}\label{toric}
Recall that by \cite[Proposition 2.11]{HK00} given a Mori Dream Space $X$ there is an embedding $i:X\rightarrow \mathcal{T}_X$ into a simplicial projective toric variety $\mathcal{T}_X$ such that $i^{*}:\Pic(\mathcal{T}_X)\rightarrow \Pic(X)$ is an isomorphism inducing an isomorphism $\Eff(\mathcal{T}_X)\rightarrow \Eff(X)$. Furthermore, the Mori chamber decomposition of $\Eff(\mathcal{T}_X)$ is a refinement of the Mori chamber decomposition of $\Eff(X)$. Indeed, if $\Cox(X) \cong \frac{K[T_1,\dots,T_s]}{I}$ where the $T_i$ are homogeneous generators with non-trivial effective $\Pic(X)$-degrees then $\Cox(\mathcal{T}_X)\cong K[T_1,\dots,T_s]$.
\end{Remark}

\begin{Notation}
We will denote by $\left\langle v_1,\dots,v_s\right\rangle$ the cone in $\mathbb{R}^n$ generated by the vectors $v_1,\dots,v_s\in\mathbb{R}^n$. Given two vectors $v_i,v_j$ we set $(v_i,v_j] := \left\langle v_i,v_j\right\rangle\setminus \mathbb{R}_{\geq 0}v_i$ and $(v_i,v_j) := \left\langle v_i,v_j\right\rangle\setminus\{\mathbb{R}_{\geq 0}v_i\cup\mathbb{R}_{\geq 0}v_j\}$.
\end{Notation}

The space $\mathcal{A}(1)$ is a point, and $\mathcal{A}(2)$, $\mathcal{A}(3)$ have Picard rank one, so there is nothing to say on their Mori chamber decomposition. Those of Picard rank two are $\mathcal{A}(4)$ and $\mathcal{A}(5)$. Note that as an immediate consequence of Theorem \ref{theff} we have that their Mori chamber decomposition coincides with their stable base locus decomposition. The stable base locus decomposition of $\Eff(\mathcal{A}(5))$ is given by the regions $[E_1^{-},D_2^{-}), [D_2^{-},D_4^{-}], (D_4^{-},D_6^{-}]$, while that of $\Eff(\mathcal{A}(4))$ is obtained from the decomposition of $\Eff(\mathcal{A}(5))$ by removing the ray $D_6^{-}\sim 3H^{-}-2E_1^{-}$. Therefore, the first interesting varieties are $\mathcal{A}(6)$ and $\mathcal{A}(7)$. We begin our investigation by first studying these decompositions for the first space $\mathcal{A}(n)_1$ appearing in Construction \ref{css}. 

\begin{Proposition}\label{firstbu}
Let $\mathcal{A}(n)_1$ be the blow-up of $\mathbb{P}^{N_{-}}$ along the Grassmannian $\mathcal{G}(1,n)\subset\mathbb{P}^{N_{-}}$ in Construction \ref{css}. As usual we write $\Pic(\mathcal{A}(n)_1) = \mathbb{Z}[H^{-},E_1^{-}]$ where $H^{-}$ is the pull-back of the hyperplane section of $\mathbb{P}^{N_{-}}$, and $E_1^{-}$ is the exceptional divisor.

Then the Mori chamber and the stable base locus decompositions of $\Eff(\mathcal{A}(n)_1)$ coincide and are represented, when $n$ is odd and even respectively, by the following pictures
$$
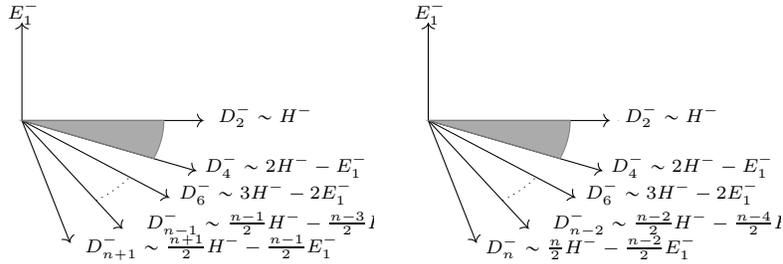
\begin{tikzpicture}[xscale=1.8,yscale=1.3][line cap=round,line join=round,>=triangle 45,x=1cm,y=1cm]\clip(-0.20150176678445222,-1.6) rectangle (2.6,1.1763586956521748);\draw [->,line width=0.4pt] (0,0) -- (1.3400182225885815,0);\draw [->,line width=0.4pt] (0,0) -- (0,1);\draw [->,line width=0.4pt] (0,0) -- (1.2849546368530922,-0.5152151776936973);\draw [->,line width=0.4pt] (0,0) -- (1.0853491385619454,-0.7905331063711422);\draw [->,line width=0.4pt] (0,0) -- (0.7412017277151399,-1.1071487243502038);\draw [->,line width=0.4pt] (0,0) -- (0.3557566275667178,-1.2516906369058625);\draw [shift={(0,0)},line width=0.4pt,dotted]  plot[domain=5.302329138917821:5.653673312285958,variable=\t]({1*0.9849158007199498*cos(\t r)+0*0.9849158007199498*sin(\t r)},{0*0.9849158007199498*cos(\t r)+1*0.9849158007199498*sin(\t r)});\draw [shift={(0,0)},line width=0.4pt,color=yqyqyq,fill=yqyqyq,fill opacity=0.66]  (0,0) --  plot[domain=-0.38133353271222603:0,variable=\t]({1*1.0480259109671113*cos(\t r)+0*1.0480259109671113*sin(\t r)},{0*1.0480259109671113*cos(\t r)+1*1.0480259109671113*sin(\t r)}) -- cycle ;\begin{scriptsize}\draw [fill=uuuuuu] (0,0) circle (0.0pt);\draw [fill=black] (1.3400182225885815,0) circle (0.0pt);\draw[color=black] (1.8,0.04769021739130457) node {$D_2^{-}\sim H^{-}$};\draw [fill=black] (0,1) circle (0.0pt);\draw[color=black] (0.011749116607773832,1.1) node {$E_1^{-}$};\draw [fill=black] (1.2849546368530922,-0.5152151776936973) circle (0.0pt);\draw[color=black] (1.95,-0.4680706521739131) node {$D_4^{-}\sim 2H^{-}-E_1^{-}$};\draw [fill=black] (1.0853491385619454,-0.7905331063711422) circle (0.0pt);\draw[color=black] (1.8,-0.7436141304347827) node {$D_6^{-}\sim 3H^{-}-2E_1^{-}$};\draw [fill=black] (0.7412017277151399,-1.1071487243502038) circle (0.0pt);\draw[color=black] (1.85,-1.0580163043478263) node {$D_{n-1}^{-}\sim \frac{n-1}{2}H^{-}-\frac{n-3}{2}E_1^{-}$};\draw [fill=black] (0.3557566275667178,-1.2516906369058625) circle (0.0pt);\draw[color=black] (1.4,-1.3) node {$D_{n+1}^{-}\sim \frac{n+1}{2}H^{-}-\frac{n-1}{2}E_1^{-}$};\end{scriptsize}\end{tikzpicture}\quad
\begin{tikzpicture}[xscale=1.8,yscale=1.3][line cap=round,line join=round,>=triangle 45,x=1cm,y=1cm]\clip(-0.20150176678445222,-1.6) rectangle (2.6,1.1763586956521748);\draw [->,line width=0.4pt] (0,0) -- (1.3400182225885815,0);\draw [->,line width=0.4pt] (0,0) -- (0,1);\draw [->,line width=0.4pt] (0,0) -- (1.2849546368530922,-0.5152151776936973);\draw [->,line width=0.4pt] (0,0) -- (1.0853491385619454,-0.7905331063711422);\draw [->,line width=0.4pt] (0,0) -- (0.7412017277151399,-1.1071487243502038);\draw [->,line width=0.4pt] (0,0) -- (0.3557566275667178,-1.2516906369058625);\draw [shift={(0,0)},line width=0.4pt,dotted]  plot[domain=5.302329138917821:5.653673312285958,variable=\t]({1*0.9849158007199498*cos(\t r)+0*0.9849158007199498*sin(\t r)},{0*0.9849158007199498*cos(\t r)+1*0.9849158007199498*sin(\t r)});\draw [shift={(0,0)},line width=0.4pt,color=yqyqyq,fill=yqyqyq,fill opacity=0.66]  (0,0) --  plot[domain=-0.38133353271222603:0,variable=\t]({1*1.0480259109671113*cos(\t r)+0*1.0480259109671113*sin(\t r)},{0*1.0480259109671113*cos(\t r)+1*1.0480259109671113*sin(\t r)}) -- cycle ;\begin{scriptsize}\draw [fill=uuuuuu] (0,0) circle (0.0pt);\draw [fill=black] (1.4,0) circle (0.0pt);\draw[color=black] (1.8,0.04769021739130457) node {$D_2^{-}\sim H^{-}$};\draw [fill=black] (0,1) circle (0.0pt);\draw[color=black] (0.011749116607773832,1.1) node {$E_1^{-}$};\draw [fill=black] (1.2849546368530922,-0.5152151776936973) circle (0.0pt);\draw[color=black] (1.95,-0.4680706521739131) node {$D_4^{-}\sim 2H^{-}-E_1^{-}$};\draw [fill=black] (1.0853491385619454,-0.7905331063711422) circle (0.0pt);\draw[color=black] (1.8,-0.7436141304347827) node {$D_6^{-}\sim 3H^{-}-2E_1^{-}$};\draw [fill=black] (0.7412017277151399,-1.1071487243502038) circle (0.0pt);\draw[color=black] (1.85,-1.0580163043478263) node {$D_{n-2}^{-}\sim \frac{n-2}{2}H^{-}-\frac{n-4}{2}E_1^{-}$};\draw [fill=black] (0.3557566275667178,-1.2516906369058625) circle (0.0pt);\draw[color=black] (1.2,-1.3) node {$D_{n}^{-}\sim \frac{n}{2}H^{-}-\frac{n-2}{2}E_1^{-}$};\end{scriptsize}\end{tikzpicture}
$$
where $\Mov(\mathcal{A}(n)_1) = \left\langle D_2^{-},D_{n-1}^{-}\right\rangle$ if $n$ is odd, while $\Mov(\mathcal{A}(n)_1) = \left\langle D_2^{-},D_{n}^{-}\right\rangle$ is $n$ is even. 
\end{Proposition}
\begin{proof}
Note that $\mathcal{A}(n)_1$ is spherical even though it is not wonderful. Arguing as in the proof of Proposition \ref{p1} we can compute the colors of $\mathcal{A}(n)_1$ which by Remark \ref{toric} are the walls of the Mori chamber decomposition.

Recall that by \cite[Section 10]{LO13} the ideal of $\sec_h(\mathcal{G}(1,n))$ is generated by the $(2h+2)\times (2h+2)$ sub-Pfaffians of the matrix $Z$ in (\ref{matrix}), and $D_{2h+2}^{-}$ is the strict transform of the hypersurface defined by such a sub-Pfaffian. By Theorem \ref{theff} the linear system of $D_{2h+2}^{-}$ becomes base-point-free when we blow-up the strict transform of $\sec_h(\mathcal{G}(1,n))$ in Construction \ref{css}, so $\textbf{B}(D_{2h+2}^{-}) = \widetilde{\sec_h(\mathcal{G}(1,n))}$, where $\widetilde{\sec_h(\mathcal{G}(1,n))}$ denotes the strict transform of $\sec_h(\mathcal{G}(1,n))\subset\mathbb{P}^{N_{-}}$ in $\mathcal{A}(n)_1$.   

Note that $\textbf{B}(D) = \emptyset$ for any $D\in [D_2^{-},D_4^{-}]$. Now, if $D$ is a $\mathbb{Q}$-divisor in $[E_1^{-},D_2^{-})$ then $\textbf{B}(D)\subset E_1^{-}$. Furthermore, if $e$ is a curve generating the extremal ray of $\NE(\mathcal{A}(n)_1)$ corresponding to the blow-down $\mathcal{A}(n)_1\rightarrow\mathbb{P}^{N_{-}}$ we have $D\cdot e < 0$, and since the curves of class $e$ cover $E_1^{-}$ we get that $E_1^{-}\subset \textbf{B}(D)$. Therefore, $\textbf{B}(D) = E_1^{-}$ for any $D\in [E_1^{-},D_2^{-})$.     

Let $D_{b_1} = D_2^{-} + b_1 E_1^{-}$, $D_{b_2} = D_2^{-} + b_2E_1^{-}$ be effective $\mathbb{Q}$-divisors in $\mathcal{A}(n)_1$ such that $b_2\leq b_1\leq 0$. Note that we can write $D_{b_2} = D_{b_1}+(b_2-b_1)E_1^{-}$, with $b_2-b_1 \leq 0$. Therefore $\textbf{B}(D_{b_1})\subset \textbf{B}(D_{b_2})$. Now, consider $D\in (D_{2h}^{-},D_{2h+2}^{-}]$, hence $D\sim D_{2}^{-}+bE_1^{-}$ with $-\frac{h}{h+1} \leq b <-\frac{h-1}{h}$. Note that if $p\in \sec_h(\mathcal{G}(1,n))$ is a general point then there is a degree $h-1$ rational normal curve $C$ in $\sec_h(\mathcal{G}(1,n))$ passing through $p$ and intersecting $\sec_h(\mathcal{G}(1,n))$ in $h$ points. Hence, if $\widetilde{C}$ is the strict transform of $C$ in $\mathcal{A}(n)_1$ then the curves in $\mathcal{A}(n)_1$ of class $\widetilde{C}$ cover $\widetilde{\sec_h(\mathcal{G}(1,n))}$. Furthermore, $\widetilde{C}\sim (h-1)l-he$, where $l$ denotes the strict transform of a general line in $\mathbb{P}^{N_{-}}$, and $D\cdot \widetilde{C} = h-1+bh <0$. Therefore, $\widetilde{\sec_h(\mathcal{G}(1,n))}\subseteq\textbf{B}(D)$, and hence $\textbf{B}(D) = \widetilde{\sec_h(\mathcal{G}(1,n))}$ for any divisor $D\in (D_{2h},D_{2h+2}]$.
\end{proof}

Thanks to Proposition \ref{firstbu} we immediately see that two different types of Sarkisov links naturally arise in our setting. For the basics on Sarkisov links we refer to \cite{Co95}. 

\begin{Proposition}\label{Sarkisov}
In the notation of Proposition \ref{firstbu} let $\mathcal{A}(n)_1^{h}$ be the model of $\mathcal{A}(n)_1$ corresponding to a divisor $D\in (D_{2h},D_{2h+2})$. If $n$ is even then the rational map $\mathbb{P}^{N_{-}}\dasharrow \mathbb{P}^n$ given by principal sub-Pfaffians of order $n$ of a general $(n+1)\times (n+1)$ skew-symmetric matrix gives rise to a Sarkisov link of type I:
 \[
  \begin{tikzpicture}[xscale=1.9,yscale=-1.1]
    \node (A0_0) at (0, 0) {};
    \node (A0_1) at (1, 0) {$\mathcal{A}(n)_1$};
    \node (A0_2) at (2, 0) {$\mathcal{A}(n)_1^{\frac{n-2}{2}}$};
    \node (A1_0) at (0, 1) {$\mathbb{P}^{N_{-}}$};
    \node (A1_2) at (2, 1) {$\mathbb{P}^n$};
    \path (A1_0) edge [->,dashed]node [auto] {$\scriptstyle{}$} (A1_2);
    \path (A0_1) edge [->]node [auto] {$\scriptstyle{}$} (A1_0);
    \path (A0_1) edge [->,dashed]node [auto] {$\scriptstyle{}$} (A0_2);
    \path (A0_2) edge [->]node [auto] {$\scriptstyle{}$} (A1_2);
    \end{tikzpicture}
  \]
If $n$ is odd then the birational involution $i:\mathbb{P}^{N_{-}}\dasharrow \mathbb{P}^{N_{-}}$ given by mapping an invertible skew-symmetric matrix $Z$ to its inverse gives rise to a Sarkisov link of type II:
\[
  \begin{tikzpicture}[xscale=1.9,yscale=-1.1]
    \node (A0_1) at (1, 0) {$\mathcal{A}(n)_1$};
    \node (A0_2) at (2, 0) {$\mathcal{A}(n)_1^{\frac{n-1}{2}}$};
    \node (A1_0) at (0, 1) {$\mathbb{P}^{N_{-}}$};
    \node (A1_3) at (3, 1) {$\mathbb{P}^{N_{-}}$};
    \path (A0_1) edge [->]node [auto] {$\scriptstyle{}$} (A1_0);
    \path (A1_0) edge [->,dashed]node [auto] {$\scriptstyle{}$} (A1_3);
    \path (A0_2) edge [->]node [auto] {$\scriptstyle{}$} (A1_3);
    \path (A0_1) edge [->,dashed]node [auto] {$\scriptstyle{}$} (A0_2);
  \end{tikzpicture}
  \]   
Furthermore, $A(n)_1^{\frac{n-1}{2}}\cong \mathcal{A}(n)_1$ and the birational involution $i:\mathbb{P}^{N_{-}}\dasharrow \mathbb{P}^{N_{-}}$ lifts to an automorphism $Z^{inv}:\mathcal{A}(n)\rightarrow \mathcal{A}(n)$.
\end{Proposition}
\begin{proof}
First, note that if $n$ is even the rational map $\mathbb{P}^{N_{-}}\dasharrow \mathbb{P}^n$ given by the $n+1$ principal sub-Pfaffians of order $n$ of a general $(n+1)\times (n+1)$ skew-symmetric matrix is dominant, and lifts to the regular fibration $\mathcal{A}(n)_1^{\frac{n-2}{2}}\rightarrow\mathbb{P}^n$ induced by $D_n^{-}$.  

The description of the maps as Sarkisov links follows from the computation of the Mori chamber decomposition of $\Eff(\mathcal{A}(n)_1)$ in Proposition \ref{firstbu}. For the second part of the statement it is enough to note that $Z^{inv}:\mathcal{A}(n)\rightarrow \mathcal{A}(n)$ is the automorphism switching $\mathbb{P}(\bigwedge^2\bigwedge^{k}V)$ and $\mathbb{P}(\bigwedge^2\bigwedge^{n+1-k}V)$ in (\ref{closure}).
\end{proof}

Now, we are ready to compute the Mori chamber and stable base locus decompositions for the spaces of complete skew-forms of Picard rank three.

\begin{thm}\label{MCD_main}
The Mori chamber decomposition of $\mathcal{A}(6)$ consists of five chambers described in the following $2$-dimensional section of $\Eff(\mathcal{A}(6))$
$$
\begin{tikzpicture}[xscale=0.7,yscale=0.5][line cap=round,line join=round,>=triangle 45,x=1cm,y=1cm]\clip(-7.6,-2.1) rectangle (4.7,4.5);\fill[line width=0pt,color=yqyqyq,fill=yqyqyq,fill opacity=0.56] (0,2) -- (-3.5,0) -- (3.52,0) -- cycle;\draw [line width=0.4pt] (-7,-2)-- (0,2);\draw [line width=0.4pt] (0,4)-- (-7,-2);\draw [line width=0.4pt] (-7,-2)-- (3.52,0);\draw [line width=0.4pt] (3.52,0)-- (0,4);\draw [line width=0.4pt] (0,4)-- (0,2);\draw [line width=0.4pt] (-3.5,0)-- (0,4);\draw [line width=0.4pt] (0,2)-- (3.52,0);\draw [line width=0.4pt] (-3.5,0)-- (3.52,0);\begin{scriptsize}\draw [fill=black] (-7,-2) circle (0.0pt);\draw[color=black] (-7.3,-1.7) node {$E_{1}^{-}$};\draw [fill=black] (0,2) circle (0.0pt);\draw[color=black] (0.3,2.2) node {$D_{4}^{-}$};\draw [fill=black] (-3.5,0) circle (0.0pt);\draw[color=black] (-3.8,0.3) node {$D_{2}^{-}$};\draw [fill=black] (3.5,0) circle (0.0pt);\draw[color=black] (3.9,0.2) node {$D_{6}^{-}$};\draw [fill=black] (0,4) circle (0.0pt);\draw[color=black] (0.08422939068100468,4.25) node {$E_{2}^{-}$};\end{scriptsize}\end{tikzpicture}
$$
where $\Mov(\mathcal{A}(6)) = \Nef(\mathcal{A}(6)) = \left\langle D_2^{-},D_4^{-},D_6^{-}\right\rangle$. The stable base locus decomposition of $\Eff(\mathcal{A}(6))$ consists of four chambers and is obtained by removing the wall joining $D_4^{-}$ with $E_2^{-}$ in the picture above.

The Mori chamber decomposition of $\mathcal{A}(7)$ consists of nine chambers described in the following $2$-dimensional section of $\Eff(\mathcal{A}(7))$
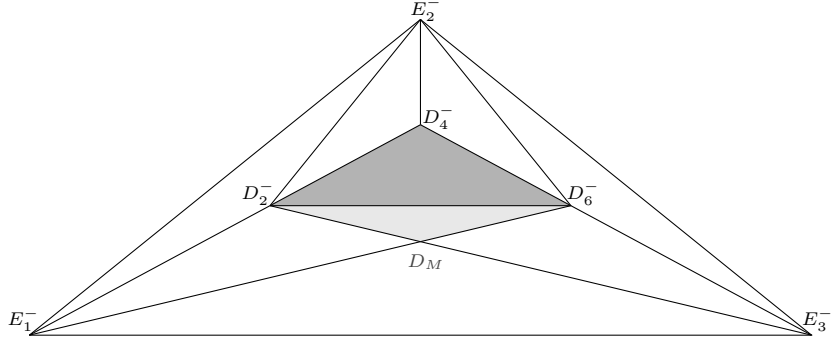
$$
\begin{tikzpicture}[xscale=0.4,yscale=0.7][line cap=round,line join=round,>=triangle 45,x=1cm,y=1cm]\clip(-13.7,-0.23) rectangle (14.1,6.4);\fill[line width=0pt,fill=black,fill opacity=0.3] (-5.000432432432432,2.4614054054054053) -- (5.000432432432432,2.4614054054054053) -- (0,4) -- cycle;\fill[line width=0pt,color=wwwwww,fill=wwwwww,fill opacity=0.15] (-5.000432432432432,2.4614054054054053) -- (5.000432432432432,2.4614054054054053) -- (0,1.7776389756402244) -- cycle;\draw [line width=0.4pt] (-13,0)-- (13,0);\draw [line width=0.4pt] (13,0)-- (0,6);\draw [line width=0.4pt] (0,6)-- (-13,0);\draw [line width=0.4pt] (0,4)-- (-13,0);\draw [line width=0.4pt] (0,4)-- (13,0);\draw [line width=0.4pt] (-5.000432432432432,2.4614054054054053)-- (13,0);\draw [line width=0.4pt] (5.000432432432432,2.4614054054054053)-- (-13,0);\draw [line width=0.4pt] (-5.000432432432432,2.4614054054054053)-- (5.000432432432432,2.4614054054054053);\draw [line width=0.4pt] (-5.000432432432432,2.4614054054054053)-- (0,6);\draw [line width=0.4pt] (0,6)-- (5.000432432432432,2.4614054054054053);\draw [line width=0.4pt] (0,4)-- (0,6);\begin{scriptsize}\draw [fill=black] (-13,0) circle (0pt);\draw[color=black] (-13.2,0.3) node {$E_1^{-}$};\draw [fill=black] (13,0) circle (0pt);\draw[color=black] (13.2,0.3) node {$E_3^{-}$};\draw [fill=black] (0,6) circle (0pt);\draw[color=black] (0.18536585365853658,6.2) node {$E_2^{-}$};\draw [fill=black] (0,4) circle (0pt);\draw[color=black] (0.6,4.143836565096953) node {$D_4^{-}$};\draw [fill=black] (-5.000432432432432,2.4614054054054053) circle (0pt);\draw[color=black] (-5.4,2.7) node {$D_2^{-}$};\draw [fill=black] (5.000432432432432,2.4614054054054053) circle (0pt);\draw[color=black] (5.4,2.7) node {$D_6^{-}$};\draw [fill=uuuuuu] (0,1.7776389756402244) circle (0pt);\draw[color=uuuuuu] (0.18536585365853658,1.4) node {$D_M$};\end{scriptsize}\end{tikzpicture}
$$
where $D_M \sim 6D_2^{-}-3E_1^{-}-2E_2^{-}$, and $\Mov(\mathcal{A}(7)) = \left\langle D_2^{-},D_4^{-},D_6^{-},D_M\right\rangle$. The stable base locus decomposition of $\Eff(\mathcal{A}(7))$ consists of eight chambers and is obtained by removing the wall joining $D_4^{-}$ with $E_2^{-}$ in the picture above.
\end{thm}
\begin{proof} 
First of all, note that by Theorem \ref{gen} the sections of $D_2^{-},D_4^{-},D_6^{-},E_1^{-},E_2^{-}$ are homogeneous generators of $\Cox(\mathcal{A}(6))$ with respect to the usual grading on $\Pic(\mathcal{A}(3))$. Furthermore, Theorems \ref{theff}, \ref{gen} and \cite[Proposition 3.3.2.3]{ADHL15} yield $\Mov(\mathcal{A}(6)) = \Nef(\mathcal{A}(6)) = \left\langle D_2^{-},D_4^{-},D_6^{-}\right\rangle$.

Now, let $\mathcal{T}_{\mathcal{A}(6)}$ be a simplicial projective toric variety as in Remark \ref{toric}. Then there is an embedding $i:\mathcal{A}(6)\rightarrow\mathcal{T}_{\mathcal{A}(6)}$ such that $i^{*}:\Pic(\mathcal{T}_{\mathcal{A}(6)})\rightarrow \Pic(\mathcal{A}(6))$ is an isomorphism inducing an isomorphism $\Eff(\mathcal{T}_{\mathcal{A}(6)})\rightarrow \Eff(\mathcal{A}(6))$. Furthermore, if we set $\widetilde{E}_j = i^{*-1}(E_j^{-}), \widetilde{D}_j = i^{*-1}(D_j^{-})$ then the sections of $\widetilde{D}_1,\widetilde{D}_2,\widetilde{D}_3,\widetilde{E}_1,\widetilde{E}_2$ are homogeneous generators of $\Cox(\mathcal{T}_{\mathcal{A}(6)})$ with respect to the grading on $\Pic(\mathcal{T}_{\mathcal{A}(6)})$ induced by the usual grading on $\Pic(\mathcal{A}(6))$ via the isomorphism $i^{*}$.

Since $\mathcal{T}_{\mathcal{A}(6)}$ is toric, the Mori chamber decomposition of $\Eff(\mathcal{T}_{\mathcal{A}(6)})$ can be computed by means of the Gelfand–Kapranov–Zelevinsky, GKZ for short, decomposition \cite[Section 2.2.2]{ADHL15}. Let us consider the family of vectors in $\Pic(\mathcal{T}_{\mathcal{A}(6)})$ given by $\mathcal{W} = (\widetilde{D}_1,\widetilde{D}_2,\widetilde{D}_3,\widetilde{E}_1,\widetilde{E}_2)$, and let $\Omega(\mathcal{W})$ be the set of all convex polyhedral cones generated by some of the vectors in $\mathcal{W}$. By \cite[Construction 2.2.2.1]{ADHL15} the GKZ chambers of $\Eff(\mathcal{T}_{\mathcal{A}(6)})$ are given by the intersections of all the cones in $\Omega(\mathcal{W})$ containing a fixed divisor $D\in\Eff(\mathcal{T}_{\mathcal{A}(6)})$. In this context a Mori chamber is the interior of a maximal cone.  

Since $\Pic(\mathcal{T}_{\mathcal{A}(6)})$ is $3$-dimensional we may picture the vectors of $\mathcal{W}$ in a $2$-dimensional section. It is straightforward to see that taking all the possible intersections of all the convex cones generated by vectors in $\mathcal{W}$ we get a picture completely analogous to the one in the statement. 

Now, Remark \ref{toric} yields that the wall-and-chamber decomposition in the statement is a possibly trivial refinement of the Mori chamber decomposition of $\Eff(\mathcal{A}(6))$. In particular, $\Mov(\mathcal{A}(6))$ coincides with $\Nef(\mathcal{A}(6))$. 

Let us analyze the stable base locus decomposition of $\Eff(\mathcal{A}(6))$. Note that $\Nef(\mathcal{A}(6)) = \left\langle D_2^{-},D_4^{-},D_6^{-}\right\rangle$ yields $\textbf{B}(D) = \emptyset$ for any $D\in \left\langle D_2^{-},D_4^{-},D_6^{-}\right\rangle$. Furthermore, note that $E_1^{-}\cup E_2^{-}$ contains $\textbf{B}(D)$ for any $D$ lying in the interior of $\left\langle D_2^{-},E_1^{-},E_2^{-}\right\rangle$ along with $(E_1^{-},E_2^{-})$. On the other hand, considering the curves described in Proposition \ref{Morimov} we see that both $E_1^{-}$ and $E_2^{-}$ are covered by curves intersecting negatively such divisor, so $\textbf{B}(D) = E_1^{-}\cup E_2^{-}$. Similarly, we can prove that $\textbf{B}(D) = E_2^{-}$ if and only if $D$ lies in the interior of the non convex union of cones $\left\langle D_2^{-},D_4^{-},E_2^{-}\right\rangle\cup \left\langle D_4^{-},D_6^{-},E_2^{-}\right\rangle$ along with $(D_6^{-},E_2]\cup (D_2^{-},E_2]\cup (D_4^{-},E_2]$. On the other hand, since Mori chambers are convex, the wall joining $D_4^{-}$ and $E_2^{-}$ must appear in the Mori chamber decomposition. For $\mathcal{A}(7)$ we can argue in an analogous way. Just note that in this case the models corresponding to the chambers $\left\langle D_2^{-},D_4^{-},E_2^{-}\right\rangle$ and  $\left\langle D_4^{-},D_6^{-},E_2^{-}\right\rangle$ are abstractly isomorphic. Indeed, they are both isomorphic to the variety $\mathcal{A}(7)_1$ in Construction \ref{css} that is to the blow-up of $\mathbb{P}^{27}$ along the Grassmannian $\mathcal{G}(1,7)$. 

In the notation of Propositions \ref{firstbu}, \ref{Sarkisov} we have the following commutative diagram
  \[
  \begin{tikzpicture}[xscale=2.5,yscale=-1.2]
    \node (A0_1) at (1, 0) {$\mathcal{A}(7)$};
    \node (A1_0) at (0, 1) {$\mathcal{A}(7)_1$};
    \node (A1_2) at (2, 1) {$\mathcal{A}(7)_{1}^{3}$};
    \node (A2_0) at (0, 2) {$\mathbb{P}^{27}$};
    \node (A2_1) at (1, 2) {$W$};
    \node (A2_2) at (2, 2) {$\mathbb{P}^{27}$};
    \path (A1_0) edge [->]node [auto] {$\scriptstyle{}$} (A2_1);
    \path (A0_1) edge [->,swap]node [auto] {$\scriptstyle{}$} (A1_0);
    \path (A1_2) edge [->]node [auto] {$\scriptstyle{}$} (A2_1);
    \path (A0_1) edge [->]node [auto] {$\scriptstyle{}$} (A1_2);
    \path (A1_0) edge [->]node [auto] {$\scriptstyle{}$} (A2_0);
    \path (A1_0) edge [->,dashed]node [auto] {$\scriptstyle{}$} (A1_2);
    \path (A1_2) edge [->]node [auto] {$\scriptstyle{}$} (A2_2);
  \end{tikzpicture}
  \]
where the rational map $\mathcal{A}(7)_1\dasharrow\mathcal{A}(7)_{1}^{3}$ is induced by the automorphism $Z^{inv}:\mathcal{A}(7)\rightarrow\mathcal{A}(7)$ in Proposition \ref{Sarkisov}. By Proposition \ref{firstbu} we have that $\mathcal{A}(7)_{1}^{3}$ is the only small $\mathbb{Q}$-factorial modification of $\mathcal{A}(7)_1$, and $\mathcal{A}(7)_1\dasharrow\mathcal{A}(7)_{1}^{3}$ is the flop associated to the small contraction induced by $D_4^{-}$. Indeed, $D_4^{-}$ has zero intersection with the strict transform of a line secant to $\mathcal{G}(1,7)$, and the strict transform of $\sec_2(\mathcal{G}(1,7))$ is the exceptional locus of the small contraction $\mathcal{A}(7)_1\rightarrow W$ induced by $D_4^{-}$. So, even though $\mathcal{A}(7)_1$ and $\mathcal{A}(7)_{1}^{3}$ are abstractly isomorphic, crossing the wall generated by $D_4^{-}$ and $E_2^{-}$ we get a non-trivial flop among them, and hence in the Mori chamber decomposition we have the additional wall $[D_4^{-},E_2^{-}]$.
\end{proof}

The arguments used in the proof of Theorem \ref{MCD_main} also apply to the space $\mathcal{A}(8)$ which has Picard rank four.

\begin{thm}\label{MCD_main_2}
The Mori chamber decomposition of $\mathcal{A}(8)$ consists of fifteen chambers and the movable cone is divided in two Mori chambers as described in the following $3$-dimensional section of $\Mov(\mathcal{A}(8))$
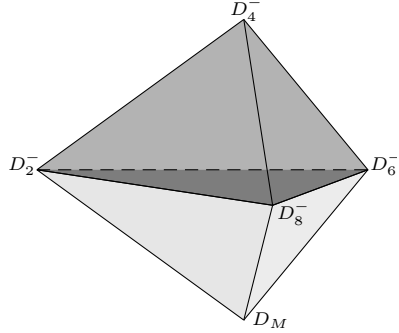
$$
\definecolor{uququq}{rgb}{0.25098039215686274,0.25098039215686274,0.25098039215686274}\begin{tikzpicture}[xscale=0.55,yscale=0.4][line cap=round,line join=round,>=triangle 45,x=1cm,y=1cm]\clip(-4.8,-5.3) rectangle (4.7,5.6);\fill[line width=0.4pt,fill=black,fill opacity=0.3] (-4,0) -- (4,0) -- (1.697368122518437,-1.181155055948025) -- cycle;\fill[line width=0pt,fill=black,fill opacity=0.3] (1.697368122518437,-1.181155055948025) -- (4,0) -- (1,5) -- cycle;\fill[line width=0pt,dash pattern=on 5pt off 3pt,fill=black,fill opacity=0.3] (1.697368122518437,-1.181155055948025) -- (1,5) -- (-4,0) -- cycle;\fill[line width=0pt,fill=black,fill opacity=0.1] (1,-5) -- (1.697368122518437,-1.181155055948025) -- (-4,0) -- cycle;\fill[line width=0pt,color=uququq,fill=uququq,fill opacity=0.1] (1,-5) -- (4,0) -- (1.697368122518437,-1.181155055948025) -- cycle;\draw [line width=0.4pt,dash pattern=on 5pt off 3pt] (-4,0)-- (4,0);\draw [line width=0.4pt] (4,0)-- (1.697368122518437,-1.181155055948025);\draw [line width=0.4pt] (1.697368122518437,-1.181155055948025)-- (-4,0);\draw [line width=0.4pt] (1,5)-- (1.697368122518437,-1.181155055948025);\draw [line width=0.4pt] (1.697368122518437,-1.181155055948025)-- (4,0);\draw [line width=0.4pt] (4,0)-- (1,5);\draw [line width=0.4pt] (1,5)-- (-4,0);\draw [line width=0.4pt] (-4,0)-- (1.697368122518437,-1.181155055948025);\draw [line width=0.4pt] (1.697368122518437,-1.181155055948025)-- (1,-5);\draw [line width=0.4pt] (1,-5)-- (4,0);\draw [line width=0.4pt] (1,-5)-- (-4,0);\begin{scriptsize}\draw [fill=black] (-4,0) circle (0.0pt);\draw[color=black] (-4.3,0.20186980609418226) node {$D_2^{-}$};\draw [fill=black] (4,0) circle (0.0pt);\draw[color=black] (4.45,0.20186980609418226) node {$D_6^{-}$};\draw [fill=black] (1.697368122518437,-1.181155055948025) circle (0.0pt);\draw[color=black] (2.2,-1.45) node {$D_8^{-}$};\draw [fill=black] (1,5) circle (0.0pt);\draw[color=black] (1.064516129032258,5.35) node {$D_4^{-}$};\draw [fill=black] (1,-5) circle (0.0pt);\draw[color=black] (1.65,-5.0) node {$D_M$};\end{scriptsize}\end{tikzpicture}
$$
where $D_M\sim 6H^{-}-3E_1^{-}-2E_2^{-}$. Furthermore, in the stable base locus decomposition of $\Eff(\mathcal{A}(8))$ there are two non convex stable base locus chambers which are the union of two Mori chambers, and other two non convex stable base locus chambers which are the union of three Mori chambers. The remaining three Mori chambers are also three different stable base locus chambers. 
\end{thm}
\begin{proof}
The arguments, from the theoretical point of view, are the ones used in the proof of Theorem \ref{MCD_main}. Clearly, in practice it is harder to figure out how the cones intersect since we are working in a $3$-dimensional space. Also in this case the Mori chamber decomposition of $\mathcal{A}(8)$ corresponds with the GKZ decomposition of the corresponding toric variety. This last claim together with all that follows can be checked by using the Magma library \texttt{SBLib.m}, see \cite[Remark 3.19]{LMR18}.

The union of the two Mori chambers 
$$
\begin{array}{l}
\left\langle (0,0,1,0),(1,0,0,0),(4,-3,-2,-1),(2,-1,0,0)\right\rangle \\ 
\left\langle (0,0,1,0),(3,-2,-1,0),(4,-3,-2,-1),(2,-1,0,0)\right\rangle
\end{array} 
$$
gives a single stable base locus chamber, and the same holds for the two Mori chambers 
$$
\begin{array}{l}
\left\langle (0,1,0,0),(0,0,0,1),(1,0,0,0),(6,-3,-2,0)\right\rangle \\ 
\left\langle (0,1,0,0),(0,0,0,1),(4,-3,-2,-1),(6,-3,-2,0)\right\rangle
\end{array} 
$$
Furthermore, the union of the three Mori chambers 
$$
\begin{array}{l}
\left\langle (0,0,0,1),(0,0,1,0),(1,0,0,0),(2,-1,0,0)\right\rangle \\ 
\left\langle (0,0,0,1),(0,0,1,0),(3,-2,-1,0),(2,-1,0,0)\right\rangle \\
\left\langle (0,0,0,1),(0,0,1,0),(4,-3,-2,-1),(3,-2,-1,0)\right\rangle
\end{array} 
$$
gives a single stable base locus chamber, and the same holds for the three Mori chambers 
$$
\begin{array}{l}
\left\langle (0,0,0,1),(1,0,0,0),(2,-1,0,0),(3,-2,-1,0)\right\rangle \\ 
\left\langle (0,0,0,1),(1,0,0,0),(6,-3,-2,0),(3,-2,-1,0)\right\rangle \\
\left\langle (0,0,0,1),(4,-3,-2,-1),(6,-3,-2,0),(3,-2,-1,0)\right\rangle
\end{array} 
$$    
The three Mori chambers
$$
\begin{array}{l}
\left\langle (0,0,1,0),(0,1,0,0),(1,0,0,0),(4,-3,-2,-1)\right\rangle \\ 
\left\langle (0,1,0,0),(1,0,0,0),(4,-3,-2,-1),(6,-3,-2,0)\right\rangle \\
\left\langle (0,0,0,1),(0,0,1,0),(0,1,0,0),(1,0,0,0)\right\rangle
\end{array} 
$$    
are also three distinct stable base locus chambers. Finally, the remaining two Mori chambers are
$$\Nef(\mathcal{A}(8)) = \left\langle (1,0,0,0),(4,-3,-2,-1),(2,-1,0,0),(3,-2,-1,0)\right\rangle$$
and $\left\langle (1,0,0,0),(4,-3,-2,-1),(6,-3,-2,0),(3,-2,-1,0)\right\rangle$, which is indeed the other chamber of $\Mov(\mathcal{A}(8))$ as displayed in the statement.
\end{proof}

\subsection{A conjecture on the stable base locus decomposition of $\Eff(\mathcal{A}(n))$}
Note that by Theorems \ref{MCD_main}, \ref{MCD_main_2} we have that
\begin{itemize}
\item[-] $\Eff(\mathcal{A}(6))\setminus\Mov(\mathcal{A}(6))$ is subdivided in $3$ stable base locus chambers, and the stable base loci are $E_1^{-},E_2^{-},E_1^{-}\cup E_2^{-}$;
\item[-] $\Eff(\mathcal{A}(7))\setminus\Mov(\mathcal{A}(7))$ is subdivided in $6$ stable base locus chambers, and the stable base loci are $E_1^{-},E_2^{-},E_3^{-},E_1^{-}\cup E_2^{-}, E_1^{-}\cup E_3^{-}, E_2^{-}\cup E_3^{-}$;
\item[-] $\Eff(\mathcal{A}(8))\setminus\Mov(\mathcal{A}(8))$ is subdivided in $7$ stable base locus chambers, and the stable base loci are $E_1^{-},E_2^{-},E_3^{-},E_1^{-}\cup E_2^{-}, E_1^{-}\cup E_3^{-}, E_2^{-}\cup E_3^{-}, E_1^{-}\cup E_2^{-}\cup E_3^{-}$.
\end{itemize}
Guided by this observation we make the following conjecture.

\begin{conj}\label{conjcham}
Assume that $n$ is even. Then $\Eff(\mathcal{A}(n))\setminus\Mov(\mathcal{A}(n))$ has 
$$\sum_{k=1}^{\frac{n-2}{2}}\binom{\frac{n-2}{2}}{k} = 2^{\frac{n-2}{2}}-1$$
stable base locus chambers, and the stable base loci are given by 
$$E_1^{-},\dots,E_{\frac{n-2}{2}}^{-},E_1^{-}\cup E_{2}^{-},\dots, E_{1}^{-}\cup\dots\cup E_{\frac{n-2}{2}}^{-}$$
Furthermore, the analogous statement holds for the spaces of complete collineations $\mathcal{X}(n,m)$ with $n < m$ in \cite[Construction 2.4]{Ma18}.

Now, assume that $n$ is odd. Then $\Eff(\mathcal{A}(n))\setminus\Mov(\mathcal{A}(n))$ has 
$$\sum_{k=1}^{\frac{n-1}{2}-1}\binom{\frac{n-1}{2}}{k} = 2^{\frac{n-1}{2}}-2$$
stable base locus chambers, and the stable base loci are given by 
$$E_1^{-},\dots,E_{\frac{n-1}{2}}^{-},E_1^{-}\cup E_{2}^{-},\dots, E_{1}^{-}\cup\dots\cup E_{\frac{n-1}{2}-1}^{-},\dots, E_{2}^{-}\cup\dots\cup E_{\frac{n-1}{2}}^{-}$$
Furthermore, the analogous statement holds for the spaces of complete collineations $\mathcal{X}(n)$ and of complete quadrics $\mathcal{Q}(n)$ in \cite[Constructions 2.4, 2.6]{Ma18}. Note that it is enough to prove the conjecture for one of the two spaces $\mathcal{X}(n)$, $\mathcal{Q}(n)$. Indeed, by \cite[Lemma 6.2]{Ma18} the conjecture for the other one follows automatically.  
\end{conj}

\subsection{Pseudo-automorphisms}\label{psaut}
Finally, a computation of the pseudo-automorphisms of $\mathcal{A}(n)$ is at hand. Recall that a pseudo-automorphism of a projective variety $X$ is a birational map $f:X\dasharrow X$ such that both $f$ and $f^{-1}$ do not contract any divisor. We will denote by $\PsAut(X)$ the group of pseudo-automorphisms of $X$.

\begin{thm}\label{pseudo-aut}
If $n\geq 4$ for $\mathcal{A}(n)$ we have
$$
\PsAut(\mathcal{A}(n)) \cong \Aut(\mathcal{A}(n)) \cong
\left\lbrace\begin{array}{ll}
PGL(n+1) & \text{if n is even}\\ 
PGL(n+1)\rtimes S_2 & \text{if n is odd}
\end{array}\right.
$$ 
while $\PsAut(\mathcal{A}(1)) \cong \Aut(\mathcal{A}(1))$ is trivial, $\PsAut(\mathcal{A}(2)) \cong \Aut(\mathcal{A}(2))\cong PGL(3)$, and finally $\PsAut(\mathcal{A}(3)) \cong \Aut(\mathcal{A}(3))\cong PGL(6)$.  
\end{thm}
\begin{proof}
Since by Corollary \ref{Fano} $\mathcal{A}(n)$ is Fano we have $\PsAut(\mathcal{A}(n))=\Aut(\mathcal{A}(n))$. Now, let $\phi\in\Aut(\mathcal{A}(n))$ be an automorphism. Then $\phi$ must act on the extremal rays of $\Nef(\mathcal{A}(n))$. If $n$ is even then Theorem \ref{theff} yields that this action must be trivial since for instance different generators of $\Nef(\mathcal{A}(n))$ have spaces of global sections of different dimensions.  

On the other hand, if $n$ is odd then either this action is trivial or it switches $D_{2j+2}^{-}$ with $D_{n-2j-1}^{-}$ for $i = 1,\dots,\frac{n-3}{2}$. We know that the latter is indeed realized by the distinguished automorphism $Z^{inv}:\mathcal{A}(n)\rightarrow\mathcal{A}(n)$ in Proposition \ref{Sarkisov}. 

Then, if $n$ is even for any automorphism $\phi$ we have in particular that $\phi^{*}D_2^{-} = D_2^{-}$, and hence via the blow-up map $f:\mathcal{A}(n)\rightarrow\mathbb{P}^{N_{-}}$ in Construction \ref{css} $\phi$ induces an automorphism $\overline{\phi}$ of $\mathbb{P}^{N_{-}}$ stabilizing the Grassmannian $\mathcal{G}(1,n)\subseteq\mathbb{P}^{N_{-}}$. To conclude it is enough to observe that since $\mathcal{A}(n)$ and $\mathbb{P}^{N_{-}}$ are birational we have $\phi = Id_{\mathcal{A}(n)}$ if and only if $\overline{\phi} = Id_{\mathbb{P}^{N_{-}}}$, and that by \cite[Theorem 1.1]{Co89} for $n\geq 4$ the group of automorphisms of $\mathbb{P}^{N_{-}}$ stabilizing $\mathcal{G}(1,n)$ is isomorphic to $PGL(n+1)$.

If $n$ is odd we have a surjective morphism $\Aut(\mathcal{A}(n))\rightarrow S_2$ where $S_2 = \{Id_{\mathcal{A}(n)}, Z^{inv}\}$. Assume that the permutation induced by $\phi\in\Aut(\mathcal{A}(n))$ is trivial. Then as before $\phi$ induces an automorphism of $\mathbb{P}^{N_{-}}$ preserving $\mathcal{G}(1,n)\subseteq\mathbb{P}^{N_{-}}$. In this case \cite[Theorem 1.1]{Co89} yields the exact sequence 
$$0\rightarrow PGL(n+1)\rightarrow\Aut(\mathcal{A}(n))\rightarrow S_2\rightarrow 0$$
where the last morphism has a section. So the sequence splits, and since the actions of $PGL(n+1)$ and $S_2 = \{Id_{\mathcal{X}(n)},Z^{inv}\}$ on $\mathcal{A}(n)$ do not commute the semi-direct product $\Aut(\mathcal{A}(n))\cong S_2 \ltimes PGL(n+1)$ is not direct. For the special cases note that $\mathcal{A}(1)$ is a point and Construction \ref{css} yields $\mathcal{A}(2)\cong\mathbb{P}^2$ and $\mathcal{A}(3)\cong\mathbb{P}^5$.
\end{proof}

\begin{Remark}
By Theorem \ref{pseudo-aut} we know that the connected component of the identity $\Aut^{o}(X)$ of $\Aut(X)$ where $X$ is any small $\mathbb{Q}$-factorial modification of $\mathcal{A}(n)$ is isomorphic to $PGL(n+1)$. Indeed, by the general theory of Mori dream spaces, see for instance Definition \ref{def:MDS}, such small modification $X$ is a Mori dream space and hence \cite[Corollary 4.2.4.2]{ADHL15} yields $\Aut^{o}(X)\cong \Aut^{o}(\mathcal{A}(n))$.
\end{Remark}

\bibliographystyle{amsalpha}
\bibliography{Biblio}

\end{document}